\documentclass{amsart}

\usepackage{amsfonts,amsmath,amsthm,amssymb,amscd,latexsym,enumerate,etoolbox,mathrsfs,comment}

\usepackage[all,cmtip,pdf]{xy}
\usepackage{thmtools, thm-restate}
\usepackage{tikz-cd}
\usepackage{tikz}
\usetikzlibrary{calc}
\usepackage{enumitem}

\newcommand{\R}{\mathbb{R}}
\newcommand{\C}{\mathbb{C}} 
\newcommand{\N}{\mathbb{N}}

\newcommand{\Z}{{\mathbb Z}}

\renewcommand{\H}{\mathcal{H}}

\newcommand{\K}{\mathcal{K}}

\newcommand{\Per}{{\bf P}}
\renewcommand{\phi}{\varphi}
\newcommand{\xn}{{\bf x}}
\newcommand{\yn}{{\bf y}}
\newcommand{\wn}{{\bf w}}
\newcommand{\zn}{{\bf z}}
\newcommand{\vn}{{\bf v}}

\theoremstyle{plain}
    \newtheorem{theorem}{Theorem}[section]
    \newtheorem{lemma}[theorem]{Lemma}
    \newtheorem{corollary}[theorem]{Corollary}
    \newtheorem{proposition}[theorem]{Proposition}
    
\theoremstyle{definition}
    \newtheorem{definition}[theorem]{Definition}
    \newtheorem{example}[theorem]{Example}

    \newtheorem{remark}[theorem]{Remark}
    
\theoremstyle{remark}

\DeclareMathOperator{\id}{id}

\DeclareMathOperator{\supp}{supp}
\DeclareMathOperator{\tr}{tr}

\DeclareMathOperator{\fin}{fin}

\begin{document}

\title[The stable algebra of a Wieler solenoid]{The stable algebra of a Wieler solenoid: inductive limits and K-theory}
\author{Robin J. Deeley}
\address{Robin J. Deeley,   Department of Mathematics,
University of Colorado Boulder
Campus Box 395,
Boulder, CO 80309-0395, USA }
\email{robin.deeley@colorado.edu}
\author{Allan Yashinski}
\address{Allan Yashinski,  Department of Mathematics, University of Maryland, College Park, MD 20742-4015, USA }
\email{ayashins@math.umd.edu}

\begin{abstract}
Wieler has shown that every irreducible Smale space with totally disconnected stable sets is a solenoid (i.e., obtained via a stationary inverse limit construction). Using her construction, we show that the associated stable $C^*$-algebra is the stationary inductive limit of a $C^*$-stable Fell algebra that has compact spectrum and trivial Dixmier-Douady invariant. This result applies in particular to Williams solenoids along with other examples. Beyond the structural implications of this inductive limit, one can use this result to in principle compute the $K$-theory of the stable $C^*$-algebra. A specific one-dimensional Smale space (the $aab/ab$-solenoid) is considered as an illustrative running example throughout.
\end{abstract}

\maketitle

\section*{Introduction}
In \cite{Wil} Williams showed that an important class of attractors can be realized via an explicit stationary inverse limit construction. These dynamical systems are examples of Axiom A basic sets \cite{Sma} and fit within Ruelle's framework of Smale spaces \cite{Rue}. Based on Williams' construction, Wieler \cite{Wie} showed that every irreducible Smale space with totally disconnected stable sets can be realized via an explicit stationary inverse limit satisfying certain natural axioms, see Section \ref{WielerSection} for the precise statement. Based on Wieler's result, we refer to such Smale spaces as Wieler solenoids.

Ruelle, Putnam, and Spielberg \cite{Put, PutSpi, Rue} have constructed and studied $C^*$-algebras associated to a Smale space. Up to Morita equivalence, the stable $C^*$-algebra of a Smale space is the groupoid $C^*$-algebra of an \'{e}tale groupoid defined using the stable equivalence relation, as in \cite{PutSpi}. In the present paper, we study the structure of this stable $C^*$-algebra for a Wieler solenoid. We use the inverse limit structure considered by Wieler to obtain a stationary inductive limit structure of the stable $C^*$-algebra.  

Our construction is very much inspired by a construction of Gon\c{c}alves \cite{Gon1, Gon2} (also see Mingo \cite{Min}, Williamson \cite{WilPhD} and the recent paper \cite{GonRamSol}) in the special case of tilings. In addition, we have been influenced by the work of Renault \cite{RenAP}, Thomsen \cite{ThoAMS, ThoSol} and Yi \cite{Yi}. Of course, the idea that an inverse limit of spaces should lead to an inductive limit of $C^*$-algebras is both natural and well-studied. Another starting point for this work is the first listed author's work with Goffeng, Mesland, and Whittaker \cite{DGMW} and with Strung \cite{DS}.

Let us briefly recall the construction of the stable $C^*$-algebra of a (mixing) Smale space $(X, \varphi)$ as defined by Putnam and Spielberg in \cite{PutSpi}. Section \ref{Section-SmaleSpaces} contains more details on their construction along with relevant definitions. We first fix a finite $\varphi$-invariant subset $\Per$ of $X$; the points in $\Per$ are periodic with respect to the homeomorphism $\varphi$.  Then we consider the set $X^u(\Per)$ of all points in $X$ which are unstably equivalent to some point in $\Per$.  On the set $X^u(\Per)$, we consider the stable equivalence relation $\sim_s$, viewed as a groupoid
\[ G^s(\Per) := \{ (x, y) \in X^u(\Per) \times X^u(\Per) \: | \: x \sim_s y \}.\]
The groupoid $G^s(\Per)$ has an \'{e}tale topology, and the stable $C^*$-algebra of $(X, \phi)$ is the groupoid $C^*$-algebra $C^*(G^s(\Per))$.

In the case where $X$ is a Wieler solenoid, we use the inverse limit structure to define a subrelation $\sim_0$ of $\sim_s$.  There is a corresponding subgroupoid
\[ G_0(\Per) := \{ (x, y) \in X^u(\Per) \times X^u(\Per) \: | \: x \sim_0 y \} \]
of $G^s(\Per)$.  The equivalence relation $\sim_0$ is constructed so that $G_0(\Per)$ is open in $G^s(\Per)$, and therefore $G_0(\Per)$ is \'{e}tale.  Building on this, we use the fact that the inverse limit is stationary to prove the following result.

\begin{theorem}
There is a nested sequence of \'{e}tale subgroupoids
\[ G_0(\Per) \subset G_1(\Per) \subset G_2(\Per) \subset \ldots \]
of $G^s(\Per)$ such that $G^s(\Per) = \bigcup_{k=0}^\infty G_k(\Per)$ and each $G_k(\Per)$ is isomorphic to $G_0(\Per)$ in a natural way.
\end{theorem}

This allows one to reduce the study of $G^s(\Per)$ to $G_0(\Per)$, which is easier to understand.  To see why it is easier, first note that the space $X^u(\Per)$ has a natural topology, which coincides with the topology of the diagonal subspace of $G^s(\Per)$. Note we never consider the subspace topology of $X^u(\Per)$ it inherits from $X$, as $X^u(\Per)$ is dense as a subset of $X$. Likewise the topology defined on $G^s(\Per)$ is not the same as the subspace topology it inherits as a subspace of $X^u(\Per) \times X^u(\Per)$. However, the topology of $G_0(\Per)$ does coincide with the subspace topology from $X^u(\Per) \times X^u(\Per)$.

This last observation places $G_0(\Per)$ into the framework of \cite{CHR, HKS}, from which it follows that $C^*(G_0(\Per))$ is a Fell algebra.  A Fell algebra can be viewed as a generalization of a continuous-trace $C^*$-algebra in which the spectrum is not necessarily Hausdorff, but is locally Hausdorff.  In particular, the notion of a Dixmier-Douady invariant can be generalized to Fell algebras. This extension was introduced in \cite{HKS} (while for more on the original notion see \cite{RaWi}).  More precisely, the results from \cite{CHR, HKS} imply that the quotient map $q: X^u(\Per) \to X^u(\Per)/{\sim_0}$ is a local homeomorphism and the $C^*$-algebra $C^*(G_0(\Per))$ is a Fell algebra with spectrum $X^u(\Per)/{\sim_0}$ and trivial Dixmier-Douady invariant.  

The nested sequence of groupoids induces an inductive limit of $C^*$-algebras $C^*(G^s(\Per)) \cong \varinjlim C^*(G_k(\Per))$.  The isomorphism $C^*(G_k(\Per)) \cong C^*(G_0(\Per))$ is such that this inductive sequence is isomorphic to a stationary inductive sequence whose connecting homomorphism $\psi: C^*(G_0(\Per)) \to C^*(G_0(\Per))$ can be described in terms of the dynamics of the Wieler solenoid.
The following theorem below is our main result.  Note that by $C^*$-stable, we mean stable in the $C^*$-algebraic sense (that is, absorbs the compact operators on a separable Hilbert space).
\begin{theorem} (see Theorem \ref{inductiveLimit}) \label{MainResult} \\
The stable $C^*$-algebra $C^*(G^s(\Per))$ of an irreducible Wieler solenoid is isomorphic to the stationary inductive limit $\varinjlim (C^*(G_0(\Per)), \psi)$ where $C^*(G_0(\Per))$ is a $C^*$-stable Fell algebra that has compact, locally Hausdorff spectrum and trivial Dixmier-Douady invariant.
\end{theorem}

The case of the tiling spaces studied by Gon\c{c}alves \cite{Gon1, Gon2} fits within the framework of this result (see \cite[Proposition 5.7]{Gon1}) and as mentioned was one of our starting points. Our main result is also a generalization of results in \cite{RenAP, ThoAMS}. In the notation of Definition \ref{WielerAxioms}, the case when $g$ is a local homeomorphism is considered in \cite{RenAP, ThoAMS}. 

One consequence of our inductive limit is that it allows one to compute the $K$-theory of the stable algebra as a stationary inductive limit of abelian groups
\[ K_*(C^*(G^s(\Per))) \cong \varinjlim (K_*(C^*(G_0(\Per))), \psi_*). \]
The main tool we use for computing $K_0(C^*(G^s(\Per)))$ as an inductive limit is a natural family of traces on $C_c(G_0(\Per))$, parametrized by the quotient space $X^u(\Per)/{\sim_0}$.  These traces are densely defined on $C^*(G_0(\Per))$, but nevertheless induce homomorphisms $K_0(C^*(G_0(\Per))) \to \Z$ because $C_c(G_0(\Per))$ is closed under holomorphic functional calculus in $C^*(G_0(\Per))$. We prove general results about these traces which allow for the computation of this inductive limit provided that the family of traces separates the elements of $K_0(C^*(G_0(\Per)))$. We illustrate the techniques for doing this in the example of the ``$aab/ab$-solenoid", a specific example of a Williams solenoid constructed from a certain map on a wedge sum of two circles.  The reader should be aware that the $K$-theory of this example (and any one-dimensional Williams solenoid) is well-known see \cite{Yi} (and also \cite{ThoSol, ThoAMS}).  Our techniques are also applicable to any one-dimensional Williams solenoid.  We would be remiss not to mention that, for higher dimensional examples, the computation of the $K$-theory from the inductive limit becomes increasingly difficult. The reader can see such examples (even in dimension two) in \cite{Gon2, GonRamSol}. 

Nevertheless, in future work, the present authors will use the techniques from the present paper to compute the $K$-theory in a number of examples. In particular, the $K$-theory of the stable algebra of the $p/q$-solenoids studied in \cite{BurkePutnam} will be computed.

Another consequence of the inductive limit structure is that it leads to a new proof of the fact that $C^*(G^s(\Per))$ has finite nuclear dimension (in the special case when the stable sets are totally disconnected). This result holds in general (that is, without any assumption on the stable sets) and is the main result of \cite{DS}.

Our inductive limit structure also has consequences on the structure of general Smale spaces (again, without any assumption on the stable sets). In work in progress, the main result of the present paper is used by the present authors and Magnus Goffeng to study the existence of projections for general Smale space $C^*$-algebras. It is worth noting that these results seems to require the full generality of Theorem \ref{MainResult}. In other words, we do not believe they can be obtained by just considering examples such as tiling spaces \cite{Gon1, Gon2, GonRamSol, Min} and the case when $g$ is a local homeomorphism (see any of \cite{DGMW, RenAP, ThoAMS})

A summary of the structure of the paper is as follows. Section 1 contains background material on general Smale spaces along with the associated stable groupoid and stable $C^*$-algebra. Wieler's results on Smale spaces with totally disconnected stable sets are reviewed in Section 2. In addition to the general theory of Wieler solenoids, a number examples are introduced and an explicit description of the stable relation is given. Sections 3, 4 and 5 contain the main theoretical results of the paper. In these sections (respectively) the subgroupoid $G_0(\Per)$ and the inductive limit decomposition of $C^*(G^s(\Per))$ are considered, the associated compact, locally Hausdorff space $X^u(\Per)/{\sim_0}$ is studied and the structure of $C^*(G_0(\Per))$ is discussed. In Sections 6 we prove that $G_0(\Per)$ has dynamic asymptotic dimension zero and hence $C^*(G_0(\Per))$ has finite nuclear dimension. It also follows that $C_c(G_0(\Per))$ is closed under holomorphic functional calculus in $C^*(G_0(\Per))$. Sections 7 and 8 contain technical results which culminate in the construction of certain traces on $C_c(G_0(\Per))$. These traces are used in Section 9 to compute $K$-theory. The main example is the ``$aab/ab$"-solenoid but some general results (e.g., Theorem \ref{Theorem-TraceOnConnectedInductiveLimit}) and other examples (e.g., Examples \ref{KtheoryNsole} and \ref{KtheorySFT}) are also considered. In the appendix, the stable and unstable $C^*$-algebras associated to a Smale space via Putnam and Spielberg's method are shown to be $C^*$-stable. This result seems to be known to experts but we could not find a reference so we have included a proof. The key tool in the proof is the main result of \cite{HjeRor}.

\section*{Acknowledgments}

The authors thank Magnus Goffeng, Bram Mesland, Ian Putnam, Adam Rennie, Aidan Sims, Karen Strung, Michael Whittaker, Rufus Willett and Robert Yuncken for interesting and insightful discussions. We also thank the referee for a number of useful suggestions.

\section{Smale spaces} \label{Section-SmaleSpaces}

\begin{definition} \label{SmaSpaDef}
A Smale space is a metric space $(X, d)$ along with a homeomorphism $\varphi: X\rightarrow X$ with the following additional structure: there exists global constants $\epsilon_X>0$ and $0< \lambda < 1$ and a continuous map, called the bracket map, 
\[
[ \ \cdot \  , \ \cdot \ ] :\{(x,y) \in X \times X : d(x,y) \leq \epsilon_X\}\to X
\]
such that the following axioms hold
\begin{itemize}
\item[B1] $\left[ x, x \right] = x$;
\item[B2] $\left[x,[y, z] \right] = [x, z]$ when both sides are defined;
\item[B3] $\left[[x, y], z \right] = [x,z]$ when both sides are defined;
\item[B4] $\varphi[x, y] = [ \varphi(x), \varphi(y)]$ when both sides are defined;
\item[C1] For $x,y \in X$ such that $[x,y]=y$, $d(\varphi(x),\varphi(y)) \leq \lambda d(x,y)$;
\item[C2] For $x,y \in X$ such that $[x,y]=x$, $d(\varphi^{-1}(x),\varphi^{-1}(y)) \leq \lambda d(x,y)$.
\end{itemize}
We denote a Smale space simply by $(X,\varphi)$.
\end{definition}

Examples of Smale spaces and an introduction to their basic properties can be found in \cite{Put}. Throughout we assume that $X$ is infinite. For the most part, the Smale spaces considered in this paper will be of a special form, which we discuss in the next section. However, we required a few general facts.

\begin{definition}
Suppose $(X, \varphi)$ is a Smale space. If $x$ and $y$ are in $X$, then we write $x \sim_s y$ (respectively, $x\sim_u y$) if $\lim_{n \rightarrow \infty} d(\varphi^n(x), \varphi^n(y)) =0$ (respectively, $\lim_{n\rightarrow \infty}d(\varphi^{-n}(x) , \varphi^{-n}(y))=0$). The $s$ and $u$ stand for stable and unstable  respectively.
\end{definition}

The global stable and unstable set of a point $x \in X$ are defined to be
\[
X^s(x) = \{ y \in X \: | \: y \sim_s x \} \hbox{ and } X^u(x)=\{ y \in X \: | \: y \sim_u x \}.
\]
Given, $0< \epsilon \le \epsilon_X$, the local stable and unstable set of a point $x \in X$ are defined respectively to be
\begin{align}
X^s(x, \epsilon) & = \{ y \in X \: | \: [x, y ]= y \hbox{ and }d(x,y)< \epsilon \} \hbox{ and } \\
X^u(x, \epsilon) & = \{ y \in X \: | \: [y, x]= y \hbox{ and } d(x,y)< \epsilon \}.
\end{align} 
The following is a standard result, see for example \cite{Put, Rue}.
\begin{theorem} \label{wellKnownSmaleSpace} Suppose $(X, \varphi)$ is a Smale space and $x$, $y$ are in $X$ with $d(x,y)< \epsilon_X$. Then the following hold: for any $0< \epsilon \le \epsilon_X$
\begin{enumerate}
\item $X^s(x, \epsilon) \cap X^u(y, \epsilon)=\{[x, y]\}$ or is empty;
\item $\displaystyle X^s(x) = \bigcup_{n\in \N} \varphi^{-n}(X^s(\varphi^n(x), \epsilon))$;
\item $\displaystyle X^u(x) = \bigcup_{n\in \N} \varphi^n(X^u(\varphi^{-n}(x), \epsilon))$.
\end{enumerate}
\end{theorem}
 A Smale space is mixing if for each pair of non-empty open sets $U$, $V$, there exists $N$ such that $\varphi^n(U)\cap V \neq \emptyset$ for all $n\geq N$. When $(X, \varphi)$ is mixing, $X^u(x)$ and $X^s(x)$ are each dense as subsets of $X$. However, one can use this theorem to give $X^u(x)$ and $X^s(x)$ locally compact, Hausdorff topologies. The details of this construction are discussed in for example \cite[Theorem 2.10]{Kil}.

Following \cite{PutSpi}, we construct the stable groupoid of $(X, \varphi)$. Let $\Per$ denote a finite $\varphi$-invariant set of periodic points of $\varphi$ and define
\[
X^u(\Per)=\{ x \in X \: | \: x \sim_u p \hbox{ for some }p \in \Per \}
\]
and
\[
G^s(\Per) := \{ (x, y) \in X^u(\Per) \times X^u(\Per) \: | \: x \sim_s y \}.
\] 
Still following \cite{PutSpi}, a topology is defined on $G^s(\Per)$ by constructing a neighborhood base. Suppose $(x,y)\in G^s(\Per)$. Then there exists $k\in \N$ such that 
\[
\varphi^k(x) \in X^s\left(\varphi^k(y), \frac{\epsilon_X}{2}\right).
\]
Since $\varphi$ is continuous there exists $\delta>0$ such that 
\[
\varphi^k( X^u(y, \delta)) \subseteq X^u\left(\varphi^k(y), \frac{\epsilon_X}{2}\right).
\]
Using this data, we define a function $h_{(x,y,\delta)} : X^u(y, \delta) \rightarrow X^u(x, \epsilon_X)$ via
\[ 
z \mapsto \varphi^{-k} ( [ \varphi^k(z) , \varphi^k(x) ]) 
\]
and have the following result from \cite{PutSpi}:
\begin{theorem} \label{etaleTopThm}
The function $h=h_{(x,y,\delta)}$ is a homeomorphism onto its image and (by letting $x$, $y$, and $\delta$ vary) the sets
\[
V(x,y,h, \delta) := \{ ( h(z), z ) \: | \: z \in X^u(y, \delta) \} 
\]
form a neighborhood base for an \'{e}tale topology on the groupoid $G^s(\Per)$. Moreover, the groupoid $G^s(\Per)$ is amenable, second countable, locally compact, and Hausdorff.
\end{theorem}

\begin{example}[An example of an open set in $G^s(\Per)$] \label{exOpenSet}
One way to construct an open set in $G^s(P)$ is to take $x$, $x' \in X^u(\Per)$ such that
\[
x' \in X^s\left(x, \frac{\epsilon_X}{2}\right)
\] 
and form
\[
V:= \left\{ ( [ y , x' ], y) \mid y \in X^s\left(x, \frac{\epsilon_X}{2}\right) \right\}.
\]
$V$ is an open set (this is the special case $k=0$ discussed in the paragraphs before this example). In fact, it is an open neighborhood of the point $(x', x) \in G^s(\Per)$. 

A further special case occurs when $x = x'$. These open sets gives the topology on the unit space of $G^s(\Per)$, which is $X^u(\Per)$ (the topology on this space is the one discussed just after Theorem \ref{wellKnownSmaleSpace}).
\end{example}

\begin{definition}
Let $C^*(G^s(\Per))$ denote the $C^*$-algebra associated to the \'{e}tale groupoid $G^s(\Per)$ (the choice of completion does not affect the $C^*$-algebra because the groupoid is amenable.) This $C^*$-algebra is called the stable algebra of $(X, \varphi)$. 
\end{definition}

A Smale space is irreducible if for each pair of non-empty open sets $U$, $V$, there exists $n$ such that $\varphi^n(U)\cap V \neq \emptyset$. We will for the most part only considered the case of mixing Smale spaces. (Recall that a Smale space is mixing if for each pair of non-empty open sets $U$, $V$, there exists $N$ such that $\varphi^n(U)\cap V \neq \emptyset$ for all $n\geq N$.) 

To generalize results from the mixing to the irreducible case one uses Smale's decomposition theorem \cite{Rue}. The $C^*$-algebraic implication of this theorem is that if $(X, \varphi)$ is an irreducible Smale space, then its stable $C^*$-algebra is isomorphic to a finite direct sum of stable $C^*$-algebras associated to mixing Smale spaces. For more on this construction see for example \cite[Section 2.5]{Kil}. This direct sum decomposition implies that in many of our results (in particular, Theorem \ref{inductiveLimit}) we need only prove the mixing case.

\section{Wieler Solenoids} \label{WielerSection}
The Smale spaces considered in this paper are all solenoids (i.e., obtained via an inverse limit construction). In this section, we review Wieler's work \cite{Wie}, which (among other results) implies that a natural class of Smale spaces are actually solenoids in a quite explicit way, see Theorem \ref{WielerTheorem} for the precise statement.
\begin{definition}[Wieler's Axioms] \label{WielerAxioms}
Let $(Y, d_Y)$ be a compact metric space, and $g: Y \rightarrow Y$ be a continuous surjective map. Then, $(Y, d_Y, g)$ satisfies Wieler's axioms if there exists constants $\beta>0$, $K\in \N^+$, and $0< \gamma < 1$ such that the following hold:
\begin{description}
\item[Axiom 1] If $x,y\in Y$ satisfy $d_Y(x,y)\le \beta$, then 
\[
d_Y(g^K(x),g^K(y))\leq\gamma^K d_Y(g^{2K}(x),g^{2K}(y)).
\]
\item[Axiom 2] For all $x\in V$ and $0<\epsilon\le \beta$
\[
g^K(B(g^K(x),\epsilon))\subseteq g^{2K}(B(x,\gamma\epsilon)).
\] 
\end{description}
\end{definition}

\begin{remark} \label{notHomeo}
We will assume that $Y$ infinite. This implies that $g$ is not a homeomorphism. This can be proved using \cite[Lemma 2.7]{DGMW} and the fact that if a compact metric space admits a positively expanding homeomorphism then it must be finite but we omit the details of the proof.
\end{remark}

\begin{definition}
Suppose $(Y,d_Y, g)$ satisfies Wieler's axioms. Then on the inverse limit space 
\[
X:= \varprojlim (Y, g) = \{ (y_n)_{n\in \N} = (y_0, y_1, y_2, \ldots ) \: | \: g(y_{i+1})=y_i \hbox{ for each }i\ge0 \}
\]
we let $\varphi: X \rightarrow X$ be defined via
\[
\varphi(x_0, x_1, x_2, \ldots ) = (g(x_0), g(x_1), g(x_2), \ldots) = (g(x_0), x_0, x_1, \ldots ).
\] 
We take the metric $d_X$ on $X$ defined via
\[
d_X((x_n)_{n\in \N}, (y_n)_{n\in \N} ) = \sum_{i=0}^K \gamma^{i} d^{\prime}_X( \varphi^{i}(x_n)_{n\in \N}, \varphi^{i}(y_n)_{n\in \N}),
\]
where $d^{\prime}_X ( (x_n)_{n\in \N}, (y_n)_{n\in \N} )= \sup_{n\in \N} (\gamma^n d_Y(x_n, y_n))$.  We refer to $(X, d_X, \varphi)$ as a Wieler solenoid. 
\end{definition}
\begin{theorem}\cite[Theorems A and B on page 4]{Wie} \label{WielerTheorem}
If $(Y, d_Y, g)$ satisfies Wieler's axioms, then the associated Wieler solenoid $(X, d_X, \varphi)$ is a Smale space with totally disconnected stable sets. The constants in Wieler's definition give Smale space constants: $\epsilon_X=\frac{\beta}{2}$ and $\lambda=\gamma$. Moreover, if $\xn=(x_n)_{n\in\N} \in X$ and $0< \epsilon \le \frac{\beta}{2}$, the locally stable and unstable sets of $(X, d_X, \varphi)$ are given by
\[
X^s( \xn , \epsilon)= \{ \yn=(y_n)_{n\in \N} \: | \: y_m= x_m \hbox{ for }0\le m \le K-1 \hbox{ and } d_X(\xn,\yn) \le \epsilon \}
\]
and
\[
X^u(\xn, \epsilon) =\{  \yn=(y_n)_{n\in \N} \: | \: d_Y(x_n,y_n)< \epsilon \: \forall n \hbox{ and } d_X(\xn,\yn) \le \epsilon \}
\]
respectively. 

Conversely, if $(X, \varphi)$ is any irreducible Smale space with totally disconnected stable sets, then there exists $(Y, d_Y, g)$ satisfying Wieler's axioms such that $(X, \varphi)$ is conjugate to the Wieler solenoid associated to $(Y, d_Y, g)$.
\end{theorem}
\begin{remark}
Wieler's axioms and the previous theorem should be compared with work of Williams \cite{Wil}. An important difference between the two is that Wieler's are purely metric space theoretic.
\end{remark}

If $g: Y \to Y$ satisfies Wieler's Axioms, then $g$ is finite-to-one by \cite[Lemma 3.4]{Wie}. In addition, given $(Y, g)$ satisfying Wieler's axioms, \cite[Theorem A on page 2068]{Wie} states that the associated Wieler solenoid is irreducible if and only if $(Y, g)$ is non-wandering and $g$ has a dense orbit in $Y$.

We will occasionally consider the special case in which $g$ is also a local homeomorphism.  This case was studied in detail in \cite{DGMW}.  In particular it was shown that if $g$ satisfies Wieler's axioms and either $g$ is open or $g^K$ is locally expanding, then $g$ is a local homeomorphism \cite[Lemmas 2.7 and 2.8]{DGMW}.

\subsection{Examples of Wieler solenoids}
A few examples of Wieler solenoids are discussed in this section. The main example we will consider in this paper is the $aab/ab$-solenoid. As such, the other examples are only discussed briefly. Nevertheless, we hope to convince the reader that many interest dynamical systems fit within Wieler's framework.

\begin{example}[Subshifts of finite type]
In \cite{Wie} (see \cite{WiePhD} for the details), Wieler shows that any irreducible two-sided subshift of finite type can be obtained by applying her construction to a suitable one-sided subshift of finite type. In this example, $g$ is the shift map, which is a local homeomorphism.  It is worth noting that Williams' solenoid construction does not apply to this example.
\end{example}

\begin{example}[$n$-solenoid]
Consider the unit circle $S^1 \subseteq \C$ with the arc length metric, rescaled so that the total circumference is $1$.  For a fixed integer $n > 1$, define
$g: S^1 \rightarrow S^1$ via $z \mapsto z^n$.
Then $(S^1, g)$ satisfies Wieler's axioms with $K = 1$, $\gamma = 1/n$, $\beta = 1/2n^2$. Hence the associated inverse limit is a Smale space. Note that $g$ is a local homeomorphism.
\end{example}

\begin{example}[$aab/ab$-solenoid]
We shall consider the following example throughout the paper.  Let $Y = S^1 \vee S^1$ be the wedge sum of two circles. Consider the following diagram:
\newpage
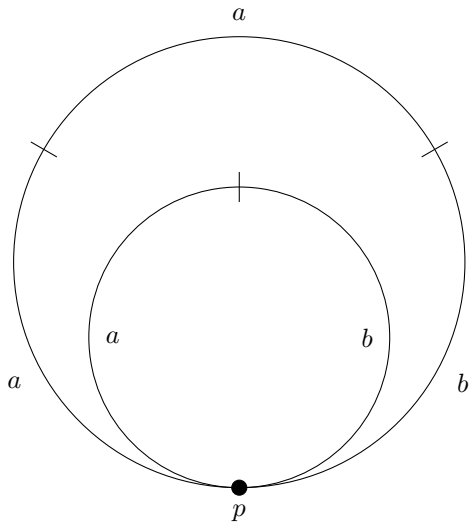
\begin{figure}[h]
\begin{tikzpicture}
\draw (0,0) circle [radius=3];
\draw (0,-1) circle [radius=2];
\draw[fill] (0,-3) circle [radius=0.1];
\draw (0,.8) -- (0,1.2);
\draw (-2.42487,1.4) -- (-2.77128,1.6);
\draw (2.42487,1.4) -- (2.77128,1.6);
\node [below] at (0,-3.1) {$p$};
\node [above] at (0,3.1) {$a$};
\node [left] at (-2.77128,-1.6) {$a$};
\node [right] at (2.77128,-1.6) {$b$};
\node [right] at (-1.9,-1) {$a$};
\node [left] at (1.9,-1) {$b$};
\end{tikzpicture}
\caption{$aab/ab$ pre-solenoid}
\label{Figure-aab/ab-PreSolenoid}
\end{figure}

The map $g: Y \rightarrow Y$ is defined as follows. We consider the outer circle to be the $a$-circle and the inner circle to be the $b$-circle. Each line segment labelled with $a$ is mapped onto the $a$-circle (i.e., the outer circle); while, each line segment labelled with $b$ is mapped onto the $b$-circle (i.e., the inner circle). The mapping is done in an orientation-preserving way, provided we've oriented both circles the same way, say clockwise. Note that $g$ is not a local homeomorphism in this example. For more details on this specific example and one-solenoids in general, see \cite{ThoSol, WilOneDim, Yi}.
\end{example}

More generally, any Williams solenoid \cite{Wil} can be studied using the constructions in the present paper (both of the two previous examples are Williams solenoids). However, there are examples of Wieler solenoids that do not fit into Williams framework (e.g., the subshifts of finite type example above).  

\begin{example}[Tilings]
As we mentioned in the introduction, one of our starting points for this paper was the work of Gon\c{c}alves \cite{Gon1, Gon2} on $C^*$-algebras associated to tiling spaces. Results in \cite{AndPut} link tiling space theory with Smale space theory. As such, there is more than one $C^*$-algebra associated to a tiling space. Gon\c{c}alves studies the stable algebra in \cite{Gon1, Gon2}, while the unstable algebra is studied in \cite{AndPut}. For more on these different algebras in the tiling space case, see \cite{GonRamSol} and references therein. In the present paper, we are interested in the stable algebra. In particular, see \cite[Section 4]{AndPut} for the construction of the relevant inverse limit in this case. Computations of the $K$-theory groups of many interesting tiling examples can be found in \cite{Gon2, GonRamSol}. 
\end{example}

\begin{example}[Gasket example]
See example 3 on pages 2070-2071 of \cite{Wie} for an interesting example of a map $g: Y \to Y$ that satisfies Wieler's axioms.  Here, the space $Y$ is formed by gluing six copies of the Sierpinski gasket together. This construction is also discussed in greater detail in \cite[Section 4.3]{WiePhD}.
\end{example}
 
\subsection{The stable relation}
The stable equivalence relation in the case of a Wieler solenoid has a particularly nice description in terms of the inverse limit structure.
\begin{lemma} \label{stableEquivIs}
Suppose $(Y, d_Y, g)$ satisfies Wieler's axioms, $(X, \varphi)$ is the associated Wieler solenoid, and $\Per$ is a finite $\varphi$-invariant set in $X$. Let $(x_n)_{n \in \N}$ and $(y_n)_{n\in \N}$ denote elements of $X^u({\bf P})$. Then, $(x_n)_{n\in \N} \sim_s (y_n)_{n\in \N}$ if and only if there exists $k\in \N$ such that $g^k(x_0)=g^k(y_0)$.
\end{lemma}
\begin{proof}
We only prove one of the two directions. Assume that $(x_n)_{n\in \N} \sim_s (y_n)_{n\in \N}$. Then there exists $N\in \N$ such that for all $n\ge N$, 
\[
d_Y(g^n(x_0), g^n(y_0))< \beta.
\]
Using this inequality, Wieler's first Axiom and a short induction argument, we have that for any positive integer $m$, 
\[
d_Y(g^{N+K}(x_0), g^{N+K}(y_0)) \le \gamma^{mK} d_Y(g^{N+(m+1)K}(x_0), g^{N+(m+1)K}(y_0))
\]
Since $\gamma^{mK}$ tends to zero as $m$ tends to $+\infty$ and 
\[
d_Y(g^{N+(m+1)K}(x_0), g^{N+(m+1)K}(y_0))< \beta
\]
we have that $g^{N+K}(x_0)=g^{N+K}(y_0)$ as required.
\end{proof}

\section{The Subrelation $\sim_0$ and Inductive Limit structure}
As above, suppose $(Y, d_Y, g)$ satisfies Wieler's axioms, $(X, \varphi)$ is the associated Wieler solenoid (which is a Smale space),  and $\Per$ is a finite $\varphi$-invariant set of points of $X$.  Elements in $X$ are denoted by 
\[ \xn = (x_0, x_1, x_2, \ldots ) = (x_n)_{n\in \N} \]
where $x_n\in Y$. It is important to note that we never consider $X^u(\Per)$ with the subspace topology it inherits from $X$.  Instead, we identify $X^u(\Per)$ with the diagonal $\{(\xn, \xn)  \: | \: \xn \in X^u(\Per) \} \subseteq G^s(\Per)$, and give $X^u(\Per)$ the subspace topology inherited from $G^s(\Per)$. We assume that $(X, \varphi)$ is mixing (see the discussion just before Section \ref{WielerSection} for more on how to generalize to the irreducible case).

Consider the map
\[ \pi_0: X^u(\Per) \to Y,\qquad \pi_0(x_0, x_1, x_2, \ldots) = x_0. \]
It follows from the structure of the local unstable sets given in Theorem \ref{WielerTheorem} that $\pi_0$ is continuous and locally injective. Since $X^u(\Per)$ is locally compact and $Y$ is Hausdorff, $\pi_0$ is locally an embedding. Moreover, it follows from the fact that $X^u(\Per)$ is dense as a subset of $X$ and $g$ is surjective that $\pi_0$ is also surjective. 

Define an equivalence relation on $X^u(\Per)$ by $\xn \sim_{\text{naive}} \yn$ if and only if $\pi_0(\xn) = \pi_0(\yn)$.  By Lemma \ref{stableEquivIs}, $\xn \sim_{\text{naive}} \yn$ implies $\xn \sim_s \yn$.  Let
\[ G_{\text{naive}}(\Per) = \left\{ (\xn, \yn) \in X^u({\Per})\times X^u({\Per}) \: | \: \pi_0(\xn) = \pi_0(\yn) \right\} \]
be the corresponding subgroupoid of $G^s(\Per)$.  The following example illustrates a fundamental problem with $G_{\text{naive}}(\Per)$: it is not open in $G^s(\Per)$ and therefore it is not \'{e}tale.

\begin{example} \label{NotOpenRel}
Consider the $aab/ab$-solenoid example. We take $\Per$ to be the set containing the single fixed point $\mathbf{p} = (p, p, p, \ldots)$.  Here, $X^u(\Per)$ is homeomorphic to the real line, and is pictured below:

\newpage

\begin{figure}[h]
\begin{tikzpicture}
\draw [thick] (-2.5,0) -- (9.5,0);
\draw (-2,-.2) -- (-2, .2);
\draw (-1,-.2) -- (-1, .2);
\draw (0,-.2) -- (0, .2);
\draw (1,-.2) -- (1, .2);
\draw (2,-.2) -- (2, .2);
\draw (3,-.2) -- (3, .2);
\draw (4,-.2) -- (4, .2);
\draw (5,-.2) -- (5, .2);
\draw (6,-.2) -- (6, .2);
\draw (7,-.2) -- (7, .2);
\draw (8,-.2) -- (8, .2);
\draw (9,-.2) -- (9, .2);
\node at (-2.5,.25){$\ldots$};
\node at (-1.5,.25){$a$};
\node at (-.5,.25){$b$};
\node at (0,-.5){$\mathbf{p}$};
\node at (.5,.25){$a$};
\node at (1, -.5){$\mathbf{q}$};
\node at (1.5,.25){$a$};
\node at (2, -.5){$\mathbf{r}$};
\node at (2.5,.25){$b$};
\node at (3.5,.25){$a$};
\node at (4.5,.25){$a$};
\node at (5.5,.25){$b$};
\node at (6.5,.25){$a$};
\node at (7.5,.25){$b$};
\node at (8.5,.25){$a$};
\node at (9.5,.25){$\ldots$};
\end{tikzpicture}
\caption{$X^u(\Per)$ for the $aab/ab$-solenoid.}
\label{Figure-aab/ab-X^u(P)}
\end{figure}

Intervals labelled with $a$ (resp. $b$) are mapped by $\pi_0$ to the outer (resp. inner) circle in $Y$. Identifying the endpoints of these intervals as $\Z$, we describe $\sim_{\text{naive}}$ as follows: All integer points are equivalent (e.g., $\mathbf{p}$, $\mathbf{q}$, and $\mathbf{r}$ are equivalent) and non-integer points $\xn$ and $\yn$ are equivalent if and only if they are in intervals with the same label and $\xn-\yn \in \Z$. That $G_{\text{naive}}(\Per)$ is not open can be seen by noting that neither $(\mathbf{p}, \mathbf{q})$ nor $(\mathbf{p}, \mathbf{r})$ has an open neighborhood in $G^s(\Per)$ contained in $G_{\text{naive}}(\Per)$.  The point $(\mathbf{p}, \mathbf{r})$ is isolated in $G_{\text{naive}}(\Per)$, whereas an open neighborhood of $(\mathbf{p}, \mathbf{q})$ in $G_{\text{naive}}(\Per)$ is homeomorphic to a half-open, half-closed interval.  Neither is open in $G^s(\Per)$, because $G^s(\Per)$ is \'{e}tale and locally homeomorphic to $X^u(\Per) \cong \R$.
\end{example}

Based on the previous example, we need to refine this naive equivalence relation  so as to obtain an open (hence \'etale) subgroupoid of $G^s(\Per)$. Essentially, we do this by excluding pairs such as $(\mathbf{p}, \mathbf{q})$ and $(\mathbf{p}, \mathbf{r})$ in the example above. This process is similar in spirit to the collaring construction done in \cite{AndPut}, though it is not the same. The process in \cite{AndPut} always outputs a Hausdorff space, while ours does not.  In the $aab/ab$-solenoid example, the collaring construction in \cite{AndPut} introduces duplicates of the $a$ tile.

Before giving the general definition of our relation, some lemmas are required.  Recall that $K, \beta, \gamma$ are the constants from Wieler's axioms. The next lemma follows  directly from the definition of the metric on $X$ so we omit the proof. (The proof is easiest to see in the case when one has $K=1$, but holds in general).

\begin{lemma} \label{KzeroLemma}
There exists $K_0 \in \N$ such that 
\begin{enumerate}
\item $K_0 \ge K$;
\item if $\xn$ and $\yn$ are in $X$ and $x_i=y_i$ for all $i\le K_0$, then $d(\xn, \yn)< \frac{\beta}{4}$.
\end{enumerate}
\end{lemma}

\begin{lemma} \label{EpsilonYLemma}
Suppose $K_0$ is a fixed natural number satisfying the conclusion of the previous lemma. Then there exists $0<\epsilon_Y< \frac{\beta}{4}$ such that for any $\xn\in X$, 
\[
\varphi^{K_0} (X^u(\xn, \epsilon_Y )) \subseteq X^u\left(\varphi^{K_0}(\xn), \frac{\beta}{4} \right)
\]
We emphasize that $\epsilon_Y$ is independent of $\xn$; it does depend on $K_0$.
\end{lemma}
\begin{proof}
By the definition of the metric on $X$, there exists $0< \epsilon < \frac{\beta}{4}$ such that if $\xn$ and $\yn$ satisfy $d_Y(x_n, y_n) < \epsilon$ for all $n \in \N$, then $d_X(\xn, \yn)< \frac{\beta}{4}$. Moreover, since $g$ is continuous and $Y$ is compact, $g$ is uniformly continuous. Hence, there exists $0< \delta \leq \epsilon$, such that for each $1\leq i \leq K_0$,
$d_Y(g^i(w), g^i(z)) < \epsilon$ whenever $d_Y(w, z) < \delta$. 

Taking $\epsilon_Y=\delta$, we have the required properties. Fix $\xn \in X$. Then, 
\[
X^u(\xn, \epsilon_Y) = \{ \yn \in X \mid d_Y(x_n,y_n)< \epsilon_Y \: \forall n \hbox{ and } d_X(\xn,\yn) \le \epsilon_Y \}.
\]
Let $\yn \in X^u(\xn, \epsilon_Y)$. Then, $d_Y(x_n, y_n)< \delta$ for each $n \in \N$. We must show $\varphi^{K_0}( \yn ) \in X^u\left(\varphi^{K_0}(\xn), \frac{\beta}{4} \right)$. By definition, 
 \[
 \varphi^{K_0}(\yn) = (g^{K_0}(y_0), g^{K_0-1}(y_0) , \dots, g(y_0), y_0, y_1, \ldots )
 \]
 and likewise
 \[
 \varphi^{K_0}(\xn) = (g^{K_0}(x_0), g^{K_0-1}(x_0) , \dots, g(x_0), x_0, x_1, \ldots )
 \]
By the construction of $\epsilon_Y$ (note: $\epsilon_Y=\delta$) and the fact that $d_Y(x_0, y_0) < \delta$, we have that, for each $1\leq i \leq K_0$, $d_Y(g^i(x_0), g^i(y_0)) < \epsilon < \frac{\beta}{4}$. Furthermore, since $\delta< \frac{\beta}{4}$, we have that $d_Y(x_n, y_n)< \frac{\beta}{4}$ for each $n\in \N$. Thus, to complete the proof, we need only show that $d(\varphi^{K_0}(\xn), \varphi^{K_0}(\yn))< \frac{\beta}{4}$. This follows from the first line of the proof and the fact that $\epsilon_Y<\delta \leq \epsilon$.
\end{proof}

We now fix $K_0$ and $\epsilon_Y$ satisfying that conditions of the previous two lemmas. Let $\pi_0 : X^u(\Per) \rightarrow Y$ denote the map defined via
\[ 
\xn = (x_n)_{n\in \N} \mapsto x_0.
\]
We can now state our main definition.
\begin{definition} \label{tildeZeroDef}
Suppose $\xn$ and $\yn$ are in $X^u(\Per)$. Then $\xn \sim_0 \yn$ if 
\begin{enumerate}
\item $\pi_0(\xn)=\pi_0(\yn)$ (i.e., $x_0 = y_0$);
\item there exists $0< \delta_{\xn} < \epsilon_Y$ and open set $U \subseteq X^u(\yn, \epsilon_Y)$ such that
\[ 
\pi_0 ( X^u(\xn, \delta_{\xn})) =\pi_0 (U). 
\]
\end{enumerate}
Let $G_0(\Per) = \{ (\xn, \yn) \: | \: \xn \sim_0 \yn \}$.
\end{definition}

Before proceeding with the general theory, let's consider what $\sim_0$ is in our main examples.

\begin{example} \label{simZeroLocalHomeo}
In the special case when $g:Y \rightarrow Y$ is a local homeomorphism, Theorem 3.12 of \cite{DGMW} implies that $\pi_0: X^u(\Per) \rightarrow Y$ is a covering map. It then follows that (1) implies (2), so we have $\xn \sim_0 \yn$ if and only if $\pi_0(\xn)=\pi_0(\yn)$.  That is, $\sim_0$ is the naive equivalence $\sim_{\text{naive}}$ in this case.
\end{example}

\begin{example}
For the $aab/ab$-solenoid, the relation $\sim_0$ can be described using the following diagram:

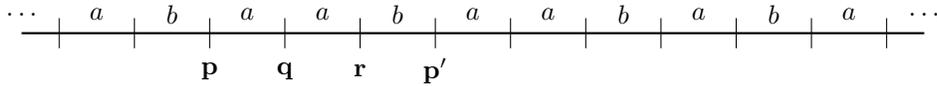
\begin{figure}[h]
\begin{tikzpicture}
\draw [thick] (-2.5,0) -- (9.5,0);
\draw (-2,-.2) -- (-2, .2);
\draw (-1,-.2) -- (-1, .2);
\draw (0,-.2) -- (0, .2);
\draw (1,-.2) -- (1, .2);
\draw (2,-.2) -- (2, .2);
\draw (3,-.2) -- (3, .2);
\draw (4,-.2) -- (4, .2);
\draw (5,-.2) -- (5, .2);
\draw (6,-.2) -- (6, .2);
\draw (7,-.2) -- (7, .2);
\draw (8,-.2) -- (8, .2);
\draw (9,-.2) -- (9, .2);
\node at (-2.5,.25){$\ldots$};
\node at (-1.5,.25){$a$};
\node at (-.5,.25){$b$};
\node at (0,-.5){$\mathbf{p}$};
\node at (.5,.25){$a$};
\node at (1, -.5){$\mathbf{q}$};
\node at (1.5,.25){$a$};
\node at (2, -.5){$\mathbf{r}$};
\node at (2.5,.25){$b$};
\node at (3, -.5){$\mathbf{p'}$};
\node at (3.5,.25){$a$};
\node at (4.5,.25){$a$};
\node at (5.5,.25){$b$};
\node at (6.5,.25){$a$};
\node at (7.5,.25){$b$};
\node at (8.5,.25){$a$};
\node at (9.5,.25){$\ldots$};
\end{tikzpicture}\caption{The $\sim_0$ relation for the $aab/ab$ solenoid.}
\label{Figure-aab/ab-Tilde_0}
\end{figure}

For the non-integer points, the relation is the same as in Example \ref{NotOpenRel}. However, two integer points are equivalent if and only if the intervals to the left and right are labelled the same. For example, $\mathbf{p}$, $\mathbf{q}$, and $\mathbf{r}$ are all inequivalent, but $\mathbf{p}$ and $\mathbf{p'}$ are equivalent.  In this example, the integer points split into three different equivalence classes, namely the equivalence classes of $\mathbf{p}$, $\mathbf{q}$, and $\mathbf{r}$.
\end{example}

Returning to the general case, if $(\xn, \yn) \in G_0(\Per)$, then we let $h_{\xn}: X^u(\xn, \delta_{\xn}) \rightarrow X^u(\yn, \epsilon_Y)$ be defined via
\[
\wn \mapsto \varphi^{-K_0} ( [ \varphi^{K_0}(\wn), \varphi^{K_0}(\yn) ]). 
\]
We note that $h_{\xn}$ is well-defined by the conditions that $K_0$ and $\epsilon_Y$ satisfy and the construction of the \'{e}tale topology on $G^s(\Per)$ discussed in and just before Theorem \ref{etaleTopThm}. Moreover, it is a homeomorphism onto its image.

\begin{lemma} \label{hxnLemma}
If $\xn \sim_0 \yn$ with $\delta_{\xn}$ satisfying the conditions of Definition \ref{tildeZeroDef}, then for any $w \in X^u(\xn , \delta_{\xn})$, $\pi_0(\wn)= \pi_0( h_{\xn}(\wn))$. Moreover, if $V \subseteq X^u(\xn, \delta_{\xn})$ is open, then $\pi_0(V) = \pi_0(h_{\xn}(V))$; we note that $h_{\xn}(V)$ is open. 
\end{lemma}
\begin{proof}
Suppose $\wn \in X^u(\xn, \delta_{\xn})$. Then there exists $\zn \in U \subseteq X^u(\yn, \epsilon_Y)$ such that 
\[
w_0 = \pi_0(\wn) = \pi_0(\zn) = z_0.
\]
On the one hand, by Lemma \ref{EpsilonYLemma}, $\varphi^{K_0}(\zn) \in X^u\left(\varphi^{K_0}(\yn) , \frac{\beta}{4}\right)$; on the other, by Theorem \ref{WielerTheorem} and Lemma \ref{KzeroLemma}, $\varphi^{K_0}(\zn) \in X^s\left(\varphi^{K_0}(\wn), \frac{\beta}{4}\right)$. Thus, using Theorem \ref{wellKnownSmaleSpace}, 
\[ 
\varphi^{K_0}(\zn) \in X^s\left(\varphi^{K_0}(\wn), \frac{\beta}{4} \right) \cap X^u\left(\varphi^{K_0}(\yn) , \frac{\beta}{4}\right)= \{ [ \varphi^{K_0}(\wn) , \varphi^{K_0}(\yn) ] \}
\]
Hence, $\varphi^{K_0}(\zn) = [ \varphi^{K_0}(\wn) , \varphi^{K_0}(\yn) ]$. It follows that $h_{\xn}(\wn)=\zn$ and that 
\[
\pi_0(h_{\xn}(\wn))=\pi_0(\zn)=z_0=w_0
\] 
as required. 

The ``Moreover" part of the lemma follows directly from the first statement.
\end{proof}
\begin{corollary} \label{UisAcutally}
If $\xn \sim_0 \yn$, then the set $h_{\xn}(X^u(\xn, \delta_{\xn}))$ satisfies the requirements of the set $U$ in the definition of $\sim_0$.
\end{corollary}
\begin{proposition}
$G_0(\Per)$ is an equivalence relation.
\end{proposition}
\begin{proof}
It is clear that $\sim_0$ is reflexive. We will show that it is also symmetric and transitive.

Suppose $\xn \sim_0 \yn$. By assumption, $x_0=y_0$. Moreover, by Lemma \ref{KzeroLemma}, 
\[
d(\varphi^{K_0}(\xn), \varphi^{K_0}(\yn)) < \frac{\beta}{4}
\]
and by Corollary \ref{UisAcutally}, the open set $U=h_{\xn}(X^u(\xn, \delta_{\xn})) \subseteq X^u(\yn, \epsilon_Y)$ satisfies the conditions required in Definition \ref{tildeZeroDef}. In particular, $\yn \in U$. Hence, there exists $0<\delta_{\yn}< \epsilon_Y$ such that $X^u(\yn, \delta_{\yn}) \subseteq U$. Since $h_{\xn}$ is a homeomorphism onto its image and $X^u(\yn, \delta_{\yn})$ is contained in its image, we have that $V:=h_{\xn}^{-1}( X^u( \yn, \delta_{\yn}))$ is an open subset of $X^u(\xn, \epsilon_Y)$. Moreover, by Lemma \ref{hxnLemma}, $\pi_0( V) = \pi_0( X^u(\yn, \delta_{\yn}))$. Thus $V$ satisfies the conditions required in the definition of $\sim_0$ and $\yn \sim_0 \xn$.

Finally, suppose $\xn \sim_0 \yn$ and $\yn \sim_0 \zn$. By the definition of $\sim_0$,
\[
x_0 = y_0 = z_0.
\] 
Since $h_{\xn}$ is a homeomorphism onto its image (which contains $\yn$) there exists $0< \delta < \delta_x$ such that 
\[
h_{\xn}( X^u(\xn, \delta)) \subseteq X^u(\yn, \delta_{\yn}).
\]
Then, $(h_{\yn} \circ h_{\xn})( X^u(\xn, \delta )) \subseteq X^u(\zn, \epsilon_Y)$ and moreover, using Lemma \ref{hxnLemma} twice,
\[
\pi_0 ( (h_{\yn} \circ h_{\xn})( X^u(\xn, \delta )) = \pi_0 ( h_{\xn}( X^u ( \xn, \delta)) = \pi_0 ( X^u(\xn, \delta)).
\]
Hence, the set $U:= (h_{\yn} \circ h_{\xn})( X^u(\xn, \delta ))$ satisfies the requirements in Definition \ref{tildeZeroDef} and $\xn \sim_0 \zn$.
\end{proof}

\begin{proposition} \label{GzeroOpenSubGrpAndEtale}
$G_0(\Per)$ is an open subgroupoid of $G^s(\Per)$. In particular, $G_0(\Per)$ is \'etale.
\end{proposition}
\begin{proof}
If $\xn \sim_0 \yn$, then $x_0 = y_0$. Lemma \ref{stableEquivIs} implies that $\xn \sim_s \yn$. To see that $G_0(\Per)$ is open, still assuming $\xn \sim_0 \yn$, we recall that $h_{\xn}$ leads to an open set in $G^s(\Per)$ by taking
\[
V:= \{ (\wn, h_{\xn}(\wn)) \: | \: \wn \in X^{u}(\xn, \delta_{\xn}) \}.
\]
Lemma \ref{hxnLemma} implies that $V \subseteq G_0(\Per)$. In more detail, the first part of Lemma \ref{hxnLemma} implies that $\pi_0(\wn)=\pi_0(h_{\xn}(\wn))$ and the second part of Lemma \ref{hxnLemma} implies the second condition in Definition \ref{tildeZeroDef} holds. 
\end{proof}

\begin{proposition}
The subspace topology on $G_0(\Per)$ obtained as a subspace of $G^s(\Per)$ coincides with the subspace topology obtained from $G_0(\Per) \subseteq X^u(\Per) \times X^u(\Per)$.
\end{proposition}
\begin{proof}
For each $k$, the topology on the set 
\[
(\varphi^{-k} \times \varphi^{-k}) \left(  \left\{ ( \xn, \yn ) \in X^u(\Per) \times X^u(\Per) \: \mid \: \yn \in X^s\left(\xn, \frac{\beta}{4} \right) \right\} \right)
\]
obtained from the topology on $G^s(\Per)$ and the subspace space topology obtained from $X^u(\Per) \times X^u(\Per)$ coincide. By Lemma \ref{KzeroLemma} and Theorem \ref{WielerTheorem}, $G_0(\Per)$ is contained in 
\[
(\varphi^{-K_0} \times \varphi^{-K_0}) \left(  \left\{ ( \xn, \yn ) \in X^u(\Per) \times X^u(\Per) \: \mid \: \yn \in X^s\left(\xn, \frac{\beta}{4} \right) \right\} \right).
\]
\end{proof}

\begin{lemma} \label{tildeZeroEquivIfKzeroEqual}
If $\xn$ and $\yn$ are in $X^u(\Per)$ and $x_i = y_i$ for $0\le i \le K_0$, then $\xn \sim_0 \yn$.
\end{lemma}
\begin{proof}
By the defining properties of $K_0$ (see the statement of Lemma \ref{KzeroLemma}), we have the following:
\begin{enumerate}
\item $x_n = y_n$ for $0\le n \le K$ where $K$ is the constant in Definition \ref{WielerAxioms};
\item $d(\xn, \yn) < \frac{\beta}{4}$.
\end{enumerate}
Theorem \ref{WielerTheorem} then implies that $\yn \in X^s\left(\xn, \frac{\beta}{4} \right)$. In particular, for some $0< \delta < \frac{\beta}{4}$, the bracket map defines a map from $X^u(\xn, \delta)$ to $X^u\left(\yn, \epsilon_Y \right)$ via
\[
\wn \mapsto [ \wn, \yn ].
\]
We denote this map by $h$ and consider $U=h(X^u(\tilde{\xn}, \delta))$, which is an open subset of $X^u\left(\tilde{\yn}, \epsilon_Y \right)$. Moreover, if $\zn \in U$, then by the definition of the bracket, $\zn \in X^s\left(\wn, \frac{\beta}{2}\right)$ and by Theorem \ref{WielerTheorem}, $z_0 = w_0$. It follows that $\pi_0(U) = \pi_0 ( X^u(\tilde{\xn}, \delta))$ and hence that $\xn \sim_0 \yn$.
\end{proof}

\begin{proposition} \label{infiniteDiscreteSet}
For any $\xn \in X^u(P)$, the equivalence class $[\xn]_0$ is an infinite discrete set.
\end{proposition}
\begin{proof} 
Since the equivalence relation $\sim_0$ is \'etale, the set $[\xn]_0$ is discrete. Using Theorem \ref{WielerTheorem} and the previous result, we have the following inclusions
\[
\varphi^{K_0}(X^s(\xn, \epsilon_X)) \cap X^u(P) \subseteq \{ \yn \in X^u(P) \: | \: y_{K_0} =x_{K_0} \}  \subseteq [\xn]_0.
\]
Moreover, the set $\varphi^{K_0}(X^s(\xn, \epsilon_X)) \cap X^u(P)$ is infinite because $(X, \varphi)$ is mixing. The result then follows.
\end{proof}

\begin{proposition}\label{Proposition-FinitelyManyEquivalenceClasses}
	Given $x \in Y$, there are finitely many equivalence classes of the form $[\xn]_0$ where $\pi_0(\xn) = x$.
\end{proposition}

\begin{proof}
	Fix $x \in Y$. Since $g: Y \to Y$ is finite-to-one (by \cite[Lemma 3.4]{Wie}) the set $g^{-K_0}\{ x\}$ is finite.
	If $\xn$ satisfies $\pi_0(\xn) = x$, then $x_{K_0} \in g^{-K_0}\{ x\}$. By Lemma \ref{tildeZeroEquivIfKzeroEqual} if $[\xn]_0 \neq [\yn]_0$, then $x_{K_0} \neq y_{K_0}$. It follows that the cardinality of $\{ [\xn]_0 \: | \:  \pi_0(\xn) = x\}$ is less than or equal to the cardinality of $g^{-K_0}(\{ x\})$, which is finite.
\end{proof}

We now use our equivalence relation $\sim_0$ to define an increasing sequence of equivalence relations.

\begin{definition}
For each $k\in \N$, we define
\[
G_k(\Per) := \{ (\xn , \yn) \in X^u(\Per) \times X^u(\Per) \mid \varphi^k(\xn) \sim_0 \varphi^k(\yn) \}.
\]
and write $\xn \sim_k \yn$ when $(\xn, \yn) \in G_k(\Per)$.
\end{definition}

In the special case of the $aab/ab$-solenoid, recall that $\sim_0$ has the form:

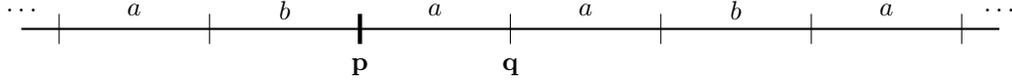
\begin{figure}[h]
\begin{tikzpicture}
\draw [thick] (-4.5,0) -- (8.5,0);
\draw (-4,-.2) -- (-4, .2);
\draw (-2,-.2) -- (-2, .2);
\draw [ultra thick] (0,-.2) -- (0, .2);
\draw (2,-.2) -- (2, .2);
\draw (4,-.2) -- (4, .2);
\draw (6,-.2) -- (6, .2);
\draw (8,-.2) -- (8, .2);
\node at (-4.5,.25){$\ldots$};
\node at (-3,.25){$a$};
\node at (-1,.25){$b$};
\node at (0,-.5){$\mathbf{p}$};
\node at (1,.25){$a$};
\node at (2,-.5){$\mathbf{q}$};
\node at (3,.25){$a$};
\node at (5,.25){$b$};
\node at (7,.25){$a$};
\node at (8.5,.25){$\ldots$};
\end{tikzpicture}
\caption{$\sim_0$ for $aab/ab$ solenoid}
\label{Figure-aab/ab-Tilde_0-Wide}
\end{figure}

\noindent Drawn on the same scale, the relation $\sim_1$ (in this special case) takes the form:

\begin{figure}[h]
\begin{tikzpicture}
\draw [thick] (-4.5,0) -- (8.5,0);
\draw (-4,-.2) -- (-4, .2);
\draw (-10/3,-.1) -- (-10/3, .1);
\draw (-8/3,-.1) -- (-8/3, .1);
\draw (-2,-.2) -- (-2, .2);
\draw (-1,-.1) -- (-1, .1);
\draw [ultra thick] (0,-.2) -- (0, .2);
\draw (2/3,-.1) -- (2/3, .1);
\draw (4/3,-.1) -- (4/3, .1);
\draw (2,-.2) -- (2, .2);
\draw (8/3,-.1) -- (8/3, .1);
\draw (10/3,-.1) -- (10/3, .1);
\draw (4,-.2) -- (4, .2);
\draw (5,-.1) -- (5, .1);
\draw (6,-.2) -- (6, .2);
\draw (20/3,-.1) -- (20/3, .1);
\draw (22/3,-.1) -- (22/3, .1);
\draw (8,-.2) -- (8, .2);
\node at (-4.5,.25){$\ldots$};
\node at (-11/3,.25){$a$};
\node at (-3,.25){$a$};
\node at (-7/3,.25){$b$};
\node at (-1.5,.25){$a$};
\node at (-.5,.25){$b$};
\node at (0,-.5){$\mathbf{p}$};
\node at (1/3,.25){$a$};
\node at (1,.25){$a$};
\node at (5/3,.25){$b$};
\node at (2,-.5){$\mathbf{q}$};
\node at (7/3,.25){$a$};
\node at (3,.25){$a$};
\node at (11/3,.25){$b$};
\node at (4.5,.25){$a$};
\node at (5.5,.25){$b$};
\node at (19/3,.25){$a$};
\node at (7,.25){$a$};
\node at (23/3,.25){$b$};
\node at (8.5,.25){$\ldots$};
\end{tikzpicture}
\caption{$\sim_1$ for $aab/ab$ solenoid}
\label{Figure-aab/ab-Tilde_1}
\end{figure}
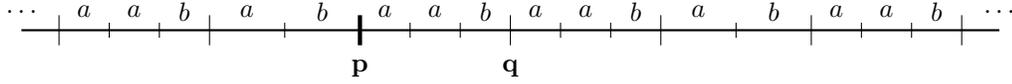

\noindent Observe that $\mathbf{p} \not\sim_0 \mathbf{q}$, but $\mathbf{p} \sim_1 \mathbf{q}$.  These two tilings of the line $X^u(\Per)$ are the same, up to homeomorphism.

Returning to the general case, the next result follows from directly from the definitions and by construction: 

\begin{proposition} \label{GkIsoGzero}
For each $k \in \N$, $G_k(\Per)$ is an open subgroupoid of $G^s(\Per)$ and $G_k(\Per) \subseteq G_{k+1}(\Per)$. Moreover, the map $\varphi^{k} \times \varphi^k : G_k(\Per) \rightarrow G_0(\Per)$ is an isomorphism of topological groupoids.
\end{proposition}

\begin{theorem} \label{inductiveLimit}
As topological groupoids,
\[ 
\bigcup_{k \in \N} G_k(\Per) = G^s(\Per),
\]
where we take the inductive limit topology on the left hand side. Moreover, $C^*(G^s(\Per))$ is isomorphic to the stationary inductive limit: 
\[
\varinjlim (C^*(G_0(\Per)), \psi)
\]
where $\psi$ is obtained as the composition of the inclusion $C^*(G_0(\Per)) \subseteq C^*(G_1(\Per))$ followed by the isomorphism $C^*(G_1(\Per)) \cong C^*(G_0(\Per))$.
\end{theorem} 
\begin{proof}
Using Proposition \ref{GkIsoGzero} (i.e., the previous proposition) the inductive limit decomposition of $C^*(G^s(\Per))$ follows once we obtain the groupoid decomposition 
\[ \bigcup_{k \in \N} G_k(\Per) = G^s(\Per). \]
Therefore, we need only show that if $\xn \sim_s \yn$, then there exists $k\in \N$ such that $\xn \sim_k \yn$. As such, suppose $\xn \sim_s \yn$. Then there exists $l \in \N$ such that $g^l(x_0) = g^l(y_0)$. Then $\varphi^{l+K_0}(\xn) \sim_0 \varphi^{l+K_0}(\yn)$ by Lemma \ref{tildeZeroEquivIfKzeroEqual} and hence that $\xn \sim_{l+K_0} \yn$. 
\end{proof}

Let us be more explicit about $\psi$, which is the composition
\[ \begin{CD}
\psi: C^*(G_0(\Per)) @>\iota>> C^*(G_1(\Per)) @>{(\varphi^{-1} \times \varphi^{-1})^*}>> C^*(G_0(\Per))
\end{CD} \]
where $\iota$ is induced by the open inclusion $G_0(\Per) \hookrightarrow G_1(\Per)$ (extending functions by zero) and the second map is the isomorphism induced by pulling back functions via the groupoid isomorphism $\varphi^{-1} \times \varphi^{-1}$.  On a function $f \in C^*(G_0(\Per))$, $\psi$ is explicitly given by
\[ \psi(f)(\xn, \yn) = f(\varphi^{-1}(\xn), \varphi^{-1}(\yn)). \]
Note that we may have $\varphi^{-1}(\xn) \sim_1 \varphi^{-1}(\yn)$ and yet $\varphi^{-1}(\xn) \not\sim_0 \varphi^{-1}(\yn)$.  Here it is understood that $f(\varphi^{-1}(\xn), \varphi^{-1}(\yn)) = 0$.

\section{The quotient space $X^u(\Per)/{\sim_0}$}

In this section, we will prove that the quotient space $X^u(\Per)/{\sim_0}$ is a compact, locally Hausdorff space, and we will establish properties of various maps between $X^u(\Per), X^u(\Per)/{\sim_0}$, and $Y$.  Before proceeding with the general theory, we consider the case when $g$ is a local homeomorphism and then our running example (i.e., the $aab/ab$-solenoid).

\begin{example} \label{localHomeoCaseSpaceLevel}
Consider the case where $g:Y \rightarrow Y$ is a local homeomorphism. As discussed in Example \ref{simZeroLocalHomeo}, $\pi_0: X^u(\Per) \rightarrow Y$ is a covering map in this case, and consequently $\xn \sim_0 \yn$ if and only if $\pi_0(\xn) = \pi_0(\yn)$. Thus $X^u(\Per)/{\sim_0} \cong Y$.
\end{example}

\begin{example}
	The $aab/ab$-solenoid is an example where $X^u(\Per)/{\sim_0}$ is not Hausdorff. The quotient space is pictured in Figure \ref{Figure-aab/ab-QuotientSpace}.  The point $p \in Y$ split into three non-Hausdorff points, denoted $ab, ba, aa$.  These points correspond to the three different $\sim_0$-equivalence classes for ``integer" points, as seen in Figure \ref{Figure-aab/ab-Tilde_0}.
	Open neighborhoods of these three points are pictured in Figure \ref{Figure-aab/ab-OpenNeighborhoods}.

\newpage	
	
	\begin{figure}[h]
		\begin{tikzpicture}
		\draw (0,0) circle [radius=3];
		\draw (0,-1) circle [radius=2];
		\draw[fill] (0,-2.6) circle [radius=0.1];
		\draw[fill] (0,-3) circle [radius=0.1];
		\draw[fill] (0,-3.4) circle [radius=0.1];
		\node [right] at (0.1,-2.7) {$ab$};
		\node [right] at (0.1,-3.2) {$ba$};
		\node [right] at (0.1,-3.6) {$aa$};
		\node [above] at (0,3.1) {$a$};
		\node [above] at (0,1.1) {$b$};
		\end{tikzpicture}
		\caption{$X^u(\Per)/{\sim_0}$ for $aab/ab$ solenoid.}
		\label{Figure-aab/ab-QuotientSpace}
	\end{figure}	
	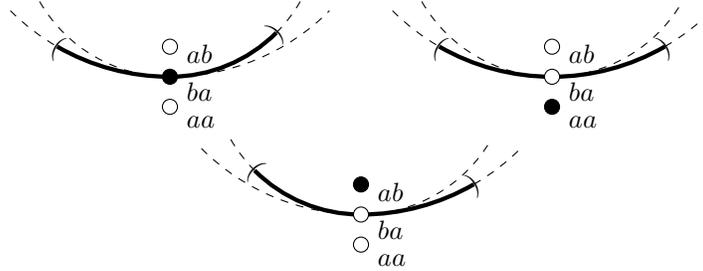
\begin{figure}[h]
		\begin{tikzpicture}
		\draw [dashed] ($(0,0) + (225:3)$) arc (225:240:3);
		\draw [ultra thick] ($(0,0) + (240:3)$) arc (240:270:3);
		\draw [dashed] ($(0,0) + (270:3)$) arc (270:315:3);
		\draw [dashed] ($(0,-1) + (210:2)$) arc (210:270:2);
		\draw [ultra thick] ($(0,-1) + (270:2)$) arc (270:315:2);
		\draw [dashed] ($(0,-1) + (315:2)$) arc (315:330:2);
		\node [rotate=-30] at (-1.5,-2.61) {$($};
		\node [rotate=45] at (1.414,-2.414) {$)$};
		\draw (0,-2.6) circle [radius=0.1];
		\draw[fill] (0,-3) circle [radius=0.1];
		\draw (0,-3.4) circle [radius=0.1];
		\node [right] at (0.1,-2.7) {$ab$};
		\node [right] at (0.1,-3.2) {$ba$};
		\node [right] at (0.1,-3.6) {$aa$};
		\end{tikzpicture} \qquad
		\begin{tikzpicture}
		\draw [dashed] ($(0,0) + (225:3)$) arc (225:240:3);
		\draw [dashed] ($(0,0) + (300:3)$) arc (300:315:3);
		\draw [ultra thick] ($(0,0) + (240:3)$) arc (240:270:3);
		\draw [ultra thick] ($(0,0) + (270:3)$) arc (270:300:3);
		\draw [dashed] ($(0,-1) + (210:2)$) arc (210:270:2);
		\draw [dashed] ($(0,-1) + (270:2)$) arc (270:330:2);
		\node [rotate=-30] at (-1.5,-2.61) {$($};
		\node [rotate=30] at (1.5,-2.61) {$)$};
		\draw (0,-2.6) circle [radius=0.1];
		\draw[fill=white] (0,-3) circle [radius=0.1];
		\draw[fill] (0,-3.4) circle [radius=0.1];
		\node [right] at (0.1,-2.7) {$ab$};
		\node [right] at (0.1,-3.2) {$ba$};
		\node [right] at (0.1,-3.6) {$aa$};
		\end{tikzpicture} \qquad
		\begin{tikzpicture}
		\draw [dashed] ($(0,-1) + (210:2)$) arc (210:225:2);
		\draw [dashed] ($(0,0) + (300:3)$) arc (300:315:3);
		\draw [dashed] ($(0,0) + (225:3)$) arc (225:270:3);
		\draw [ultra thick] ($(0,0) + (270:3)$) arc (270:300:3);
		\draw [ultra thick] ($(0,-1) + (225:2)$) arc (225:270:2);
		\draw [dashed] ($(0,-1) + (270:2)$) arc (270:330:2);
		\node [rotate=-45] at (-1.414,-2.414) {$($};
		\node [rotate=30] at (1.5,-2.61) {$)$};
		\draw[fill] (0,-2.6) circle [radius=0.1];
		\draw[fill=white] (0,-3) circle [radius=0.1];
		\draw (0,-3.4) circle [radius=0.1];
		\node [right] at (0.1,-2.7) {$ab$};
		\node [right] at (0.1,-3.2) {$ba$};
		\node [right] at (0.1,-3.6) {$aa$};
		\end{tikzpicture}
		\caption{Open neighborhoods of the three non-Hausdorff points in $X^u(\Per)/{\sim_0}$ for the $aab/ab$ solenoid.}
		\label{Figure-aab/ab-OpenNeighborhoods}
	\end{figure}
\end{example}

\begin{proposition}
	The quotient map $q: X^u(\Per) \to X^u(\Per)/{\sim_0}$ is a local homeomorphism.
\end{proposition}

\begin{proof}
The groupoid $G_0(\Per)$ is \'etale by Proposition \ref{GzeroOpenSubGrpAndEtale}. Hence \cite[Lemma 4.2]{CHR} implies that $q$ is a local homeomorphism.
\end{proof}

\begin{corollary}
	The quotient space $X^u(\Per)/{\sim_0}$ is locally metrizable.  In particular, it is locally Hausdorff.
\end{corollary}

\begin{proof} 
	The space $X^u(\Per)$ is metrizable because its topology is locally compact and Hausdorff, see for example \cite[Theorem 2.10]{Kil}. The result follows because $q$ is a local homeomorphism.
\end{proof}

\begin{definition}
	We let $r: X^u(\Per)/{\sim_0} \rightarrow Y$ be the map defined via
	\[
	[(x_i)_{i\in \N}] \mapsto x_0
	\]
\end{definition}
Note that $r$ is well-defined by the first condition in the definition of $\sim_0$. Moreover, $r\circ q = \pi_0: X^u(\Per) \rightarrow Y$ where $q$ denotes the quotient map $X^u(\Per) \rightarrow X^u(\Per)/{\sim_0}$.  By Proposition \ref{Proposition-FinitelyManyEquivalenceClasses}, $r$ is finite-to-one.  Since $q$ is a local homeomorphism and $\pi_0$ is a local embedding, it follows that $r$ is a local embedding.
\begin{definition}
	We let $\tilde{g} : X^u(\Per)/{\sim_0} \rightarrow X^u(\Per)/{\sim_0}$ be the map defined via
	\[
	[\xn]_0 \mapsto [\varphi(\xn)]_0, \qquad \text{equivalently} \qquad [(x_i)_{i\in \N}]_0 \mapsto [(g(x_i))_{i\in \N}]_0 
	\]
\end{definition}
\begin{theorem}
	The map $\tilde{g}$ defined in the previous definition is a finite-to-one, surjective, local homeomorphism. Moreover, it fits into the following commutative diagram:
	\begin{center}
		$\begin{CD}
		X^u(\Per) @>\varphi >> X^u(\Per) \\
		@Vq VV @Vq VV \\
		X^u(\Per)/{\sim_0} @>\tilde{g} >> X^u(\Per)/{\sim_0} \\
		@Vr VV @Vr VV \\
		Y @>g>> Y
		\end{CD}$
	\end{center}
	
\end{theorem}
\begin{proof}
	If $\xn \sim_0 \yn$, then $\xn \sim_1 \yn$ because $G_0(\Per) \subseteq G_1(\Per)$.  So $\varphi(\xn) \sim_0 \varphi(\yn)$ and consequently $\tilde{g}$ is well-defined.  The commutativity of the diagram is immediate. The stated properties of $\tilde{g}$ follow from the commutativity of the diagram and the properties of the other maps: $\varphi$ is a homeomorphism, $q$ is a surjective local homeomorphism, and $r$ and $g$ are finite-to-one.
\end{proof}

\begin{example} \label{Example-aab/ab-InducedMapOnQuotient}
	Recall that the quotient space for the $aab/ab$-solenoid has the form:
	
	\begin{figure}[h]
		\begin{tikzpicture}
		\draw (0,0) circle [radius=3];
		\draw (0,-1) circle [radius=2];
		\draw[fill] (0,-2.6) circle [radius=0.1];
		\draw[fill] (0,-3) circle [radius=0.1];
		\draw[fill] (0,-3.4) circle [radius=0.1];
		\node [right] at (0.1,-2.7) {$ab$};
		\node [right] at (0.1,-3.2) {$ba$};
		\node [right] at (0.1,-3.6) {$aa$};
		\draw (0,.8) -- (0,1.2);
		\draw (-2.42487,1.4) -- (-2.77128,1.6);
		\draw (2.42487,1.4) -- (2.77128,1.6);
		\node [above] at (0,3.1) {$a$};
		\node [left] at (-2.77128,-1.6) {$a$};
		\node [right] at (2.77128,-1.6) {$b$};
		\node [right] at (-1.9,-1) {$a$};
		\node [left] at (1.9,-1) {$b$};
		\node [below] at (0,0.8) {$v$};
		\node [right] at (-2.42487,1.3) {$s$};
		\node [left] at (2.42487,1.3) {$t$};
		\end{tikzpicture}
		\caption{$X^u(\Per)/{\sim_0}$ for $aab/ab$ solenoid.}
		\label{Figure-aab/ab-tilde{g}}
	\end{figure}
The map $\tilde{g}$ is exactly like $g$ for all points excluding $s, t, v, ab, ba, aa$.  For these points,
\[ \tilde{g}(s) = aa,\quad \tilde{g}(t) = \tilde{g}(v) = ab, \quad \tilde{g}(ab) = \tilde{g}(ba) = \tilde{g}(aa) = ba. \]
\end{example}

We need some lemmas before proving $X^u(\Per)/{\sim_0}$ is compact.

\begin{lemma} \label{CompactSubsetAndOnto}
	There exists $K_1>0$ such that the map 
	\[ \pi_0 \big\rvert_{\varphi^{K_1} \left( \overline{X^u\left(\Per, \frac{\beta}{4} \right)} \right)}: \varphi^{K_1} \left( \overline{X^u\left(\Per, \frac{\beta}{4} \right)} \right) \rightarrow Y \]
	is onto.
\end{lemma}
\begin{proof}
	Since $(X, \varphi)$ is mixing, $X^u(\Per)$ is dense in $X$ and hence, the map $\pi_0\big\rvert_{X^u(\Per)}$ is onto. Secondly, using Theorem \ref{wellKnownSmaleSpace} and the fact that $\Per$ is $\varphi$-invariant, $X^u(\Per) = \bigcup_{k\in \N} \varphi^{k} (X^u(\Per, \frac{\beta}{2}))$. For each $y\in Y$, let $k_y$ be the smallest natural number such that 
	\[
	B\left(y, \frac{\beta}{4}\right) \subseteq \pi_0\left( \varphi^{k_y}\left( X^u\left(\Per, \frac{\beta}{4}\right)\right)\right) 
	\]
	Then $\left\{ B\left(y, \frac{\beta}{4}\right) \right\}_{y\in Y}$ is an open cover of $Y$ and since $Y$ is compact it has a finite subcover, which we denote by $\left\{ B\left(y_i , \frac{\beta}{4}\right) \right\}_{i=1}^m$. Then $K_1 = \max \{ k_{y_i} \mid 1 \le i \le m\}$ has the required property.
\end{proof}
\begin{lemma} \label{KtwoExistLemma}
	There exists $K_2>0$ such that for any $\xn \in X^u(P)$, there exists $\widetilde{\xn} \in \varphi^{K_2} \left( X^u\left(\Per, \frac{\beta}{2}\right) \right)$ and $\xn \sim_0 \widetilde{\xn}$.
\end{lemma}
\begin{proof}
	Let $K_1$ be as in the previous lemma. We show $K_2=K_0 + K_1+1$ has the required property. Suppose $\xn \in X^u(P)$. Then $x_{K_0} \in Y$ and there exists $\yn \in X^u(\Per, \frac{\beta}{2})$ such that $\pi_0(\varphi^{K_1}( \yn )) = x_{K_0}$. It follows that 
	\[
	\widetilde{\xn} := \left(x_0, x_1, \ldots, x_{K_0}, g^{K_1-1}(y_0) , \ldots , g(y_0), y_0, y_1, \ldots \right)
	\]
	is an element of $\varphi^{K_2} (X^u(P, \frac{\beta}{2}))$. Moreover, by Lemma \ref{tildeZeroEquivIfKzeroEqual}, $\widetilde{\xn} \sim_0 \xn$. Thus $K_2$ has the required property.
\end{proof}

The previous result somewhat informally implies that the relation $\sim_0$ is ``local"; more precisely, we have the following result:
\begin{corollary} \label{locCor}
	There exists an \'etale equivalence relation on $X^u(\Per, \epsilon_X)$ with respect to the subspace topology such that the associated $C^*$-algebra is Morita equivalent to $C^*(G_0(\Per))$.
\end{corollary}

In specific cases the local unstable set (i.e., $X^u(\Per, \epsilon_X)$) is quite tractable. For example, in the case of a Williams solenoid it is always a finite disjoint union of open balls in Euclidean space.

\begin{proposition}
	The topological space $X^u(\Per)/{\sim_0}$ is compact.
\end{proposition}
\begin{proof}
	Let $\{ U_{\alpha} \}_{\alpha \in \Lambda}$ be an open cover of $X^u(\Per)/{\sim_0}$. We show it has a finite subcover. As above, $q : X^u(\Per) \rightarrow X^u(\Per) /{\sim_0}$ denotes the quotient map. By the definition of the quotient topology, $\{ q^{-1} ( U_{\alpha} ) \}_{\alpha \in \Lambda}$ is an open cover of $X^u(\Per)$. Suppose $K_2>0$ satisfies the conditions of the previous lemma and note that $\{ q^{-1} ( U_{\alpha} ) \}_{\alpha \in \Lambda}$ is an open cover of the compact set $\varphi^{K_2} ( \overline{X^u(\Per, \frac{\beta}{2})} )$. As such there is a finite subcover of $\varphi^{K_2} ( \overline{X^u(\Per, \frac{\beta}{2})} )$, we denote it by $\{ q^{-1} ( U_{\alpha_i} ) \}_{i=1}^m$. 
	
	We show that $\{ U_{\alpha_i} \}_{i=1}^m$ covers $X^u(\Per)/{\sim_0}$. Suppose $[\xn] \in X^u(\Per)/{\sim_0}$. By the previous lemma there exists $\widetilde{\xn} \in \varphi^{K_2} ( X^u(\Per, \frac{\beta}{2}) )$ and $\widetilde{\xn} \sim_0 \xn$. Since $\{ q^{-1} ( U_{\alpha_i} ) \}_{i=1}^m$ covers $\varphi^{K_2} ( \overline{X^u(\Per, \frac{\beta}{2})} )$, there exists $\alpha_{i_0}$ such that $\widetilde{\xn} \in q^{-1}(U_{\alpha_{i_0}})$. Hence $[\xn]_0= [\widetilde{\xn}]_0 \in U_{\alpha_{i_0}}$ as required.
	
\end{proof}

\section{The structure of $C^*(G_0(\Per))$}
We now consider the structure of $C^*(G_0(\Per))$ in more detail.

\begin{theorem}
The $C^*$-algebra, $C^*(G_0(\Per))$ is a Fell algebra with trivial Dixmier-Douady invariant and spectrum $X^u(\Per)/{\sim_0}$. 
\end{theorem}

\begin{proof}
The quotient map $q: X^u(\Per) \rightarrow X^u(\Per)/{\sim_0}$ is a local homeomorphism. Therefore $G_0(\Per)$ is the groupoid of an equivalence relation induced by a local homeomorphism, as in \cite{CHR}.  The statement then follows from Theorem 6.1 of \cite{CHR}. 
\end{proof}

\begin{corollary} \label{Corollary-HausdorffC*algebra}
Suppose that $X^u(\Per)/{\sim_0}$ is Hausdorff. Then, we have that 
\[ C^*(G_0(\Per)) \cong C(X^u(\Per)/{\sim_0})\otimes \mathcal{K}(\mathcal{H}).\]
\end{corollary}
\begin{proof}
Since $X^u(\Per)/{\sim_0}$ is Hausdorff, $C^*(G_0(\Per))$ is a continuous-trace $C^*$-algebra with trivial Dixmier-Douady invariant. Using this and the fact that $X^u(\Per)/{\sim_0}$ is compact, we have that $C^*(G_0(\Per))$ is Morita equivalence to $C(X^u(\Per)/{\sim_0})$. Finally, $C^*(G_0(\Per))$ is stable (see Appendix \ref{Section-C*-Stability}), which implies the result.
\end{proof}

\begin{example}
If $g: Y \rightarrow Y$ is a local homeomorphism, then by Example \ref{localHomeoCaseSpaceLevel} $X^u(\Per)/{\sim_0}=Y$ and the previous corollary then implies that in this special case $C^*(G_0(\Per)) \cong C(Y) \otimes \mathcal{K}(\mathcal{H})$.
\end{example}

Returning to the case of an arbitrary Wieler solenoid, we have the following:
\begin{theorem} \label{idealStructure}
There exist finitely many ideals $I_1, \ldots, I_N$ of $C^*(G_0(\Per))$ that generate $C^*(G_0(\Per))$, each of which has the form 
\[ I_i \cong C_0(X_i) \otimes \mathcal{K}(\mathcal{H})
\]
for some locally compact Hausdorff space $X_i$.
\end{theorem}

\begin{proof}
The space $X^u(\Per)/ \sim_0$ is compact and locally Hausdorff. As such, there exists a finite open cover $\{ U_i \}_{i=1}^N$ such that $U_i$ is Hausdorff in the subspace topology. Let $q: X^u(\Per) \rightarrow X^u(\Per)/{\sim_0}$ denote the quotient map. Using the definition of the quotient topology, $\{ q^{-1}(U_i) \}_{i=1}^N$ is an open cover of $X^u(\Per)$. Moreover, for each $i$, $q^{-1}(U_i)$ is a $G_0(\Per)$-invariant subset. 

General properties of \'etale groupoids imply that, for each $i$, $C^*(G_0(\Per)|_{q^{-1}(U_i)})$ is an ideal in $C^*(G_0(\Per))$. Using the fact that $U_i$ is Hausdorff for each $i$, we have that $C^*(G_0(\Per)|_{q^{-1}(U_i)})$ is a continuous-trace $C^*$-algebra with spectrum $U_i$ and trivial Dixmier-Douady invariant. Moreover, since they are ideals in a stable $C^*$-algebra, they are also stable.

\end{proof}

\begin{remark}
The open cover $\{ U_i \}_{i=1}^N$ of $X^u(\Per)/{\sim_0}$ in the previous theorem can be taken to have a special form: using the same method as in Corollary \ref{locCor}, one can show that each open set in the cover can be taken to be homeomorphic to an open subset of $X^u(\Per, \epsilon_X)$. 
\end{remark}

\begin{example} \label{coverAABABcase}
For the $aab/ab$-solenoid, the open cover discussed the previous theorem and remark can be taken to be the three open sets in Figure \ref{Figure-aab/ab-OpenNeighborhoods} along with two more open intervals, each one is a circle with the points $aa, ab, ba$ removed.
\end{example}

\section{Dynamic asymptotic dimension, nuclear dimension, and holomorphic functional calculus}

Here we will show that the \'{e}tale groupoid $G_0(\Per)$ has dynamic asymptotic dimension zero, as defined in \cite{GWY}. It follows that there is a bound on the nuclear dimension of $C^*(G_0(\Per))$ and that the dense subalgebra $C_c(G_0(\Per))$ is closed under holomorphic functional calculus in $C^*(G_0(\Per))$. Only the latter of these results will be used in the rest of the paper. However, it follows from the former and the inductive limit in Theorem \ref{inductiveLimit} that $C^*(G^s(\Per))$ has finite nuclear dimension. Although this is a special case of the main result of \cite{DS}, it illustrates how properties of $C^*(G_0(\Per))$ pass to $C^*(G^s(\Per))$. For more on nuclear dimension and its importance see \cite{WinterZacharias} and \cite{WinterStructure}. One can also see \cite{DS} for more on dynamic asymptotic dimension and nuclear dimension in the context of Smale spaces.

The groupoid $G_0(\Per)$ is an example of a groupoid of a local homeomorphism, as studied in \cite{CHR}.  The results of this section hold in this more general setting.  We temporarily abandon our Smale space to consider the following situation.  Let $X$ be a second countable, locally compact, Hausdorff topological space.  Let $Y$ be another topological space and let $q: X \to Y$ be a surjective local homeomorphism.  Then $q$ determines an equivalence relation $\sim_q$ on $X$ defined by $x \sim_q x'$ if and only if $q(x) = q(x')$.  Let
\[ R(q) = \{ (x, x') \in X \times X \mid q(x) = q(x')\} \]
be the corresponding groupoid, topologized as a subspace of $X \times X$.  As shown in \cite{CHR}, $R(q)$ is locally compact, Hausdorff, principal, and \'{e}tale.  Observe that our main groupoid $G_0(\Per)$ is $R(q)$ for $q: X^u(\Per) \to X^u(\Per)/{\sim_0}$.

\begin{lemma} ~
	\begin{enumerate}
		\item Let $K \subseteq X$ be compact.  There exists $N \in \N$ such that for all $x \in X$, the set $\{x' \in K \mid x \sim_q x'\}$ contains at most $N$ elements.
		\item Let $\widehat{K} \subseteq R(q)$ be compact.  There exists $M \in \N$ such that for all $x \in X$, the set $\{(x', x'') \in \widehat{K} \mid x \sim_q x' \sim_q x''\}$ contains at most $M$ elements.
	\end{enumerate}
\end{lemma}

\begin{proof}
	Since $q$ is a local homeomorphism and $K$ is compact, we can cover $K$ with finitely many open sets, say $U_1, \ldots , U_N$, such that the restriction $q|_{U_i}$ is a homeomorphism onto its image.  Given $x \in X$, it follows that each $U_i$ contains at most one member of the equivalence class $[x]$, and this proves (1).
	
	For (2), Let $s,r: R(q) \to X$ denote the source and range maps for the groupoid, which are just the projections onto each coordinate.  Let $K_1 = s(\widehat{K})$ and $K_2 = r(\widehat{K})$, both of which are compact subsets of $X$.  Let $N_1$ and $N_2$ denote the constants obtained by applying (1) to $K_1$ and $K_2$ respectively, and let $M = N_1N_2$.  If $(x', x'') \in \widehat{K}$ and $x \sim_q x' \sim_q x''$, then there are at most $N_1$ possibilities for $x'$ and $N_2$ possibilities for $x''$.
\end{proof}

\begin{lemma}
	Let $\widehat{K} \subseteq R(q)$ be compact.  Then the subgroupoid generated by $\widehat{K}$ is compact.  
\end{lemma}

\begin{proof}
	Without loss of generality, assume $\widehat{K}$ is closed under inverses and units.  Then the subgroupoid generated by $\widehat{K}$ is $\bigcup_{n=1}^\infty \widehat{K}^n$.  Note that $\widehat{K}^n$ is compact for each $n$.  However, if $M$ is as in the previous lemma, then we see that $\bigcup_{n=1}^\infty \widehat{K}^n = \bigcup_{n=1}^M \widehat{K}^n$ is compact.
\end{proof}

As defined in \cite{GWY}, a groupoid has dynamic asymptotic dimension zero if and only if it is the union of open, relatively compact subgroupoids.

\begin{proposition}
	The groupoid $R(q)$ has dynamic asymptotic dimension zero.  Consequently, the nuclear dimension of $C^*_r(R(q))$ is at most the topological covering dimension of $X$.
\end{proposition}

\begin{proof}
	Since $R(q)$ is locally compact, we can write $R(q) = \bigcup V_\alpha$ where each $V_\alpha$ is a relatively compact open subset.  Let $G_\alpha$ be the subgroupoid generated by $V_\alpha$.  Note $G_\alpha$ is open by \cite[Lemma 5.2]{GWY}.  Clearly, $R(q) = \bigcup G_\alpha$.  Further, $G_\alpha$ is relatively compact as it is contained within the subgroupoid generated by $\overline{V_\alpha}$, which is compact by the previous lemma.
	
	The assertion about nuclear dimension follows immediately from \cite[Theorem 8.6]{GWY}.  Note that we assumed $X$ to be second countable, which implies that strong dynamic asymptotic dimension coincides with dynamic asymptotic dimension.
\end{proof}

\begin{proposition} 
	The dense subalgebra $C_c(R(q)) \subseteq C^*_r(R(q))$ is closed under the holomorphic functional calculus.  Consequently, this inclusion induces an isomorphism $K_0(C_c(R(q))) \cong K_0(C^*_r(R(q)))$.
\end{proposition}

In the statement above, $K_0(C_c(R(q)))$ denotes the algebraic $K_0$-group of $C_c(R(q))$.  Note that $C_c(R(q))$ is not a $C^*$-algebra, in general.  Also, recall that the operator $K_0$-group of a $C^*$-algebra is isomorphic to its algebraic $K_0$-group.

\begin{proof}
	Let $f \in C_c(R(q))$.  We will show that the $C^*$-subalgebra generated by $f$ is contained in $C_c(R(q))$.  Observe that
	\[ \supp(f^*) = (\supp f)^{-1},\qquad \supp(f*g) \subseteq (\supp f)(\supp g) \]
	for any other $g \in C_c(R(q))$.  Let $H$ be the subgroupoid generated by $\supp f$, which is compact by the previous lemma.  It follows that every function in the $*$-subalgebra generated by $f$ has support contained within $H$.
	
	Since $R(q)$ is \'{e}tale, the inclusion $C_c(R(q)) \to C_0(R(q))$ extends to a continuous inclusion $C^*_r(R(q)) \to C_0(R(q))$, see \cite[Proposition 4.2]{Ren}.  It follows that the set of functions in $C_c(R(q))$ whose support is contained within $H$ is closed in $C^*_r(R(q))$.  Consequently, the $C^*$-subalgebra generated by $f$ is contained within $C_c(R(q))$.  Since $C^*$-algebras are closed under holomorphic functional calculus, this proves that $C_c(R(q))$ is closed under holomorphic functional calculus.
\end{proof}

Returning to our groupoid $G_0(\Per)$, we have proved the following.

\begin{theorem} \label{Theorem-DADNucDimK0}Let $N$ denote the topological covering dimension of $X^u(\Per)$.
	\begin{enumerate}
		\item The groupoid $G_0(\Per)$ has dynamic asymptotic dimension zero.
		\item The nuclear dimension of $C^*(G_0(\Per))$ is at most $N$.
		\item The nuclear dimension of $C^*(G^s(\Per))$ is at most $N$.
		\item Inclusion induces an isomorphism $K_0(C_c(G_0(\Per))) \cong K_0(C^*(G_0(\Per)))$.
	\end{enumerate}
\end{theorem}

Note that the third assertion follows from our inductive limit structure and the general behavior of nuclear dimension with respect to inductive limits \cite{WinterZacharias}.  Bounds on the nuclear dimension of $C^*$-algebras associated to general Smale spaces were first obtained in \cite{DS}. The bound on nuclear dimension obtained in \cite{DS} also used dynamic asymptotic dimension.

\section{Further remarks on the structure of the spaces}
The relationships between the spaces $X^u(\Per)$, $X^u(\Per)/{\sim_0}$ and $Y$ along with the maps between them are discussed further in this section. Recall there is a commutative diagram
\[ \begin{CD}
X^u(\Per) @>\varphi >> X^u(\Per) \\
@Vq VV @Vq VV \\
X^u(\Per)/{\sim_0} @>\tilde{g} >> X^u(\Per)/{\sim_0} \\
@Vr VV @Vr VV \\
Y @>g>> Y
\end{CD}\]
The map $q: X^u(\Per) \rightarrow X^u(\Per)/{\sim_0}$ is a local homeomorphism, but it is not a covering map in general.  We prove a weaker result, Theorem \ref{WeakerThanCoverMapResult}, which shows that $q$ has properties reminiscient of a covering map.

\begin{lemma} \label{HausNbhdInside}
If $y\in Y$ with $r^{-1}(y)=\{ r_1, \ldots , r_l \}$ and $\{ U_i \}_{i=1}^{l}$ is a collection of Hausdorff neighborhoods of the $r_i$'s, then there exists $\delta>0$,  such that for each $0< \delta' \le \delta$,
\[ r^{-1}(B(y, \delta')) \subseteq \bigcup_{i=1}^{l} U_i.
\]
\end{lemma}

\begin{proof}
For each $i=1, \ldots l$, let $U_i$ be a Hausdorff neighborhood of $r_i$. Suppose that no $\delta>0$ exists. Then there is a sequence $(w_s)_{s\in\N}$ in $X^u(\Per)/{\sim_0}$ such that $(r(w_s))_{s\in \N}$ converges to $y$ but for all $s$, $w_s \notin  \bigcup_{i=1}^{l} U_i$. Using Lemma \ref{CompactSubsetAndOnto}, there exists a sequence $({\bf x_s})_{s\in \N}$ in $X^u(\Per)$ such that $q({\bf x_s})=w_s$ for each $s$ and $({\bf x_s})_{s\in \N}$ is contained in a compact subset of $X^u(P)$. 

Hence, there exists a convergent subsequence, $({\bf x_{s_k}})_{k\in \N}$. Let $\xn$ denote its limit and note that $\pi_0(\xn)=y$. It follows that $q(\xn) \in r^{-1}(\{y\})$ and that for $k$ large enough $q({\bf x_{s_k}})=w_{s_k}$ is an element of $\bigcup_{i=1}^{l} U_i$. This is a contradiction. 
\end{proof}

\begin{lemma} \label{localHomeoToNbhd}
For each $\xn \in X^u(\Per)$ there exists $\delta>0$ such that for each open set $U\subseteq X^u(\xn, \delta)$, we have that for each ${\bf x_i} \sim_0 \xn$ there exists maps
\[ h_i : U \rightarrow X^u({\bf x_i}, \epsilon_Y) 
\]
such that for each $i$,
\begin{enumerate}[label=(\arabic*)]
\item \label{lemma-part1} $h_i$ is a homeomorphism onto its image and maps $\xn$ to ${\bf x_i}$;
\item \label{lemma-part2} $q|_U = (q \circ h_i)|_U = q|_{h_i(U)}$;
\item \label{lemma-part3} $q|_U$ is a homeomorphism onto its image in $X^u(P)/{\sim_0}$;
\item \label{lemma-part4} $h_i(X^u(\xn, \delta)) \cap h_j(X^u(\xn, \delta)) =\emptyset$ whenever ${\bf x_i} \neq {\bf x_j}$.
\end{enumerate}
\end{lemma}

\begin{proof}
Since $q$ is a local homeomorphism, condition \ref{lemma-part3} can be obtained by taking $\delta$ small. 

To begin, we prove the result for a single ${\bf x'} \sim_0 \xn$. By the definition of $\sim_0$ and Lemma \ref{hxnLemma}, there exists $\delta>0$ (depending on both $\xn$ and ${\bf x'}$) and map $h : X^u(\xn, \delta) \rightarrow X^u({\bf x'}, \epsilon_Y)$ such that
\begin{enumerate}[label=(\roman*)]
\item $h$ is a homeomorphism onto its image and maps $\xn$ to ${\bf x'}$;
\item For each $\zn\in X^u(\xn, \delta))$, we have $\pi_0(\zn)=\pi_0(h(\zn))$.
\end{enumerate}
In fact, because the second condition in the definition of $\sim_0$ is an open condition, we have that $\zn \sim_0 h(\zn)$ so that $q(\zn)=q(h(\zn))$ for each $\zn\in X^u(\xn, \delta)$. 

An induction argument implies that we have that conditions \ref{lemma-part1}-\ref{lemma-part4} hold for any finite set $\{ {\bf x_1}, \ldots , {\bf x_l} \}$ where for each $i$, ${\bf x_i} \sim_0 \xn$.

Continuing, we consider another special case. Suppose ${\bf x'} \in X^u(P)$ and $\pi_{K_0}({\bf x'})=\pi_{K_0}(\xn)$. Note that in particular that it follows that ${\bf x'} \sim_0 \xn$ and $d({\bf x'}, \xn)< \epsilon_Y$ (see Lemmas \ref{KzeroLemma} and \ref{tildeZeroEquivIfKzeroEqual}). We will show that there is a $\delta>0$ (independent of the choice of ${\bf x'}$) and $h : X^u(\xn, \delta) \rightarrow X^u({\bf x'}, \epsilon_Y)$ such that
\begin{enumerate}[label=(\roman*)]
\item \label{h-part1} $h$ is a homeomorphism onto its image and maps $\xn$ to ${\bf x'}$;
\item \label{h-part2} For each $\zn \in X^u(\xn, \delta))$, we have $\pi_0(\zn)=\pi_0(h(\zn))$.
\end{enumerate}
For $h$, we take the map $X^u(\xn, \delta) \rightarrow X^u({\bf x'}, \epsilon_Y)$ via
\[
\zn \mapsto [\zn, {\bf x'}].
\]
The definition and properties of bracket for a Wieler solenoid (see \cite[Lemmas 3.1 and 3.3]{Wie}) imply that $h$ has properties \ref{h-part1} and \ref{h-part2}. Again, because the second condition in the definition of $\sim_0$ is an open condition, we have that $\zn \sim_0 h(\zn)$ so that $q(\zn)=q(h(\zn))$ for each $\zn\in X^u(\xn, \delta)$. This implies that conditions \ref{lemma-part1}-\ref{lemma-part3} in the statement of the lemma hold in this special case. Finally, condition \ref{lemma-part4} holds since \cite[Lemma 3.3]{Wie} implies that if $\pi_{K_0}(\zn)=\pi_{K_0}({\bf z'})$ for $\zn$, ${\bf z'}$ in $X^u(\xn, \epsilon_Y)$, then $\zn={\bf z'}$. 

Given $\xn$, we denote the associated constant by $\delta_{\xn, K_0}$ to emphasize its dependence on $\xn$ and $K_0$ and its independence from ${\bf x'}$. 

We now prove the general case. Fix $\xn \in X^u(\Per)$. Since the maps $g: Y \rightarrow Y$ and $r: X^u(P)/{\sim_0} \rightarrow Y$ are each finite-to-one, there exists a finite set $\mathcal{F} \subseteq [\xn]_0$ such that if ${\bf x'} \sim_0 \xn$, then there exists ${\bf \hat{x}} \in \mathcal{F}$ such that $\pi_{K_0}({\bf \hat{x}}) = \pi_{K_0}({\bf x'})$. 

Let $\delta_1>0$ be the constant obtained by applying the special case of a finite subset to the set $\mathcal{F}$. For each ${\bf \hat{x}} \in \mathcal{F}$ we have $\delta_{{\bf \hat{x}}, K_0}>0$. Take $\delta>0$ such that
\begin{enumerate}[label=(\alph*)]
\item \label{delta-part1} $\delta< \delta_1$;
\item \label{delta-part2} $\delta< \min_{{\bf \hat{x}} \in \mathcal{F}} \{ \delta_{{\bf \hat{x}}, K_0} \}$;
\item \label{delta-part3} For each ${\bf \hat{x}} \in \mathcal{F}$, the map $h : X^u(x, \delta) \rightarrow X^u({\bf \hat{x}}, \epsilon_X)$ has image contained in $X^u(\xn,  \delta_{{\bf \hat{x}}, K_0})$.
\end{enumerate}
Suppose $U \subseteq X^u(\xn, \delta)$ is an open set and ${\bf \tilde{x}} \sim_0 \xn$. Then by construction there exists ${\bf \hat{x}} \in \mathcal{F}$ such that $\tilde{x}_{K_0}=\hat{x}_{K_0}$. Hence, there exists 
$h_{{\bf \hat{x}}} : X^u({\bf \hat{x}}, \delta_{{\bf \hat{x}}, K_0}) \rightarrow X^u({\bf \tilde{x}}, \epsilon_X)$ that satisfies the statement of the lemma. Moreover, since ${\bf \hat{x}} \in \mathcal{F}$ there exists $h: X^u(\xn, \delta_1) \rightarrow X^u({\bf \hat{x}}, \epsilon_X)$ that satisfies the statement of the lemma. 

We will show that $h_{{\bf \hat{x}}} \circ h : U \subseteq X^u(\xn, \delta) \rightarrow X^u({\bf \tilde{x}}, \epsilon_X)$ satisfies conditions \ref{lemma-part1}-\ref{lemma-part4} in the statement of the lemma. To see this note that $h_{{\bf\hat{x}}} \circ h$ is well-defined by the properties \ref{delta-part1}-\ref{delta-part3} (i.e., the defining properties of $\delta$) and it is the composition of two local homeomorphisms. This implies condition \ref{lemma-part1} holds. That conditions \ref{lemma-part2} and \ref{lemma-part3} hold follows from the two special cases discussed above. Finally, condition \ref{lemma-part4} can be obtained by possibly decreasing $\delta$ further (note that there exists a global constant $\hat{\epsilon}>0$ such that if $x_0=\tilde{x}_0$ and $\xn \in X^u({\bf \tilde{x}}, \hat{\epsilon})$, then $\xn={\bf \tilde{x}}$.)

\end{proof}

\begin{theorem} \label{WeakerThanCoverMapResult}
Given $y \in Y$ with $r^{-1}(\{y\})= \{ r_1, ... , r_l\}$ there is an open neighborhood $W$ of $y$ in Y, open neighborhoods $U_i$ of each $r_i$ in $X^u(\Per)/{\sim_0}$ and $\{ V_{\xn}\}_{\xn \in \pi_0^{-1}(\{y\})}$ is a collection of disjoint open sets in $X^u(\Per)$ such that 
\begin{enumerate}[label=(\arabic*)]
\item \label{theorem-part1} $r^{-1}(W) = \bigcup_{i=1}^{l} U_i$; 
\item \label{theorem-part2} $\xn \in V_{\xn}$ for each $\xn$;
\item \label{theorem-part3} $\pi_0^{-1}(W) = \bigcup_{\xn \in \pi_0^{-1}(\{y\})} V_{\xn}$;  
\item \label{theorem-part4} For each $\xn$, $q|_{V_{\xn}}$ is a homeomorphism onto its image and moreover this image is $U_i$ where $q(\xn)=r_i$. 
\end{enumerate}
\end{theorem}

\begin{proof}
Fix $y\in Y$ with $r^{-1}(\{y\})= \{ r_1, ... , r_l\}$. Since $q$ is onto, for each $i=1, \ldots, l$, we can find ${\bf x_i} \in X^u(\Per)$ such that $q({\bf x_i})=r_i$. Since $q$ is a local homeomorphism and $X^u(\Per)/{\sim_0}$ is locally Hausdorff, there exists a collection of disjoint open sets $\{ \tilde{V}_i \}_{i=1}^{l}$ such that
\begin{enumerate}[label=(\roman*)]
\item \label{Vi-part1} ${\bf x_i} \in \tilde{V}_i$ for each $i$;
\item \label{Vi-part2} $q|_{\tilde{V}_i}$ is a homeomorphism onto its image;
\item \label{Vi-part3} For each $i$, the open set $\tilde{U}_i := q|_{\tilde{V}_i}(\tilde{V}_i)$ is a Hausdorff neighborhood of $r_i$. 
\item \label{Vi-part4} The conclusion of Lemma \ref{localHomeoToNbhd} applies to each $\tilde{V}_i$. 
\end{enumerate}
Condition \ref{Vi-part3} implies that we can apply Lemma \ref{HausNbhdInside} to $\{ \tilde{U}_i \}_{i=1}^{l}$ and hence find $\delta>0$ such that 
\[ 
r^{-1}(B(y, \delta)) \subseteq \bigcup_{i=1}^l \tilde{U}_i.
\]
For each $i=0, \ldots, l$, let $U_i := r^{-1}(B(y, \delta)) \cap \tilde{U}_i$. For each $\xn \in \pi_0^{-1}(\{y\}))$ there exists a unique $i$ such that $q(\xn)=q({\bf x_i})=r_i$ and by \ref{Vi-part4} there exists homeomorphism onto its image, $h|_{\tilde{V}_i}: \tilde{V}_i \rightarrow X^u(\xn, \epsilon_Y)$. We let $V_{\xn} := h|_{\tilde{V}_i}((q|_{\tilde{U}_i})^{-1} (U_i))$ where we have used \ref{Vi-part3} to ensure that $(q|_{\tilde{U}_i})^{-1}$ is well-defined.

We must show that the set $W:= B(y, \delta)$ along with collections $\{ U_i \}_{i=1}^l$ and $\{ V_{\xn} \}_{\xn \in \pi_0^{-1}(\{y\}))}$ satisfy \ref{theorem-part1}-\ref{theorem-part4}. Condition \ref{theorem-part1} holds since $r^{-1}(B(y, \delta)) \subseteq \bigcup_{i=1}^l \tilde{U}_i$ and the way we defined $U_i$. Condition \ref{theorem-part2} holds since given $x$ the local homeomorphism $h$ maps ${\bf x_i}$ to $\xn$. 

The proof of condition \ref{theorem-part3} is as follows:
\[
\pi_0^{-1}(W) = (r \circ q)^{-1}(W) = q^{-1} \left( \bigcup_{i=1}^l U_i \right) = \bigcup_{i=1}^l q^{-1}(U_i)
\]
For each $i$, $q^{-1}(U_i) = \bigcup_{\xn \in q^{-1}(r_i) } V_{\xn}$ by Lemma \ref{localHomeoToNbhd} (in particular, we are using condition \ref{lemma-part2} in the statement of that lemma). Condition \ref{theorem-part3} now follows by taking the union over $i=0, \ldots, l$.

Finally for condition \ref{theorem-part4}, fix $\xn \in X^u(\Per)$ such $\pi_0(\xn)=y$. It follows that $q(\xn)=q({\bf x_i})=r_i$ (for some fixed unique $i$). When considering Lemma \ref{localHomeoToNbhd} for points that are equal rather than just equivalent we can take $h$ to be the identity map. Hence, $V_{{\bf x_i}}=(q|_{\tilde{U}_i})^{-1} (U_i)$ from which it follows that $q|_{V_{{\bf x_i}}}$ is a homeomorphism onto its image, which is $U_i$. For $V_{\bf x} = h|_{\tilde{V}_i}((q|_{\tilde{U}_i})^{-1} (U_i))$, we have the result since $q \circ h|_{\tilde{V}_i} =q|_{\tilde{U_i}}$ by condition \ref{lemma-part2} in the statement of Lemma \ref{localHomeoToNbhd}.

\end{proof}

\section{Traces on $C_c(G_0(\Per))$} \label{Section-Traces}

In this section, we will consider a family of traces on $C_c(G_0(\Per))$ which are parametrized by the points of $X^u(\Per)/{\sim_0}$.  These traces do not extend to the $C^*$-algebra $C^*(G_0(\Per))$, but the induced maps on $K$-theory can be viewed as homomorphisms defined on $K_0(C^*(G_0(\Per)))$ by Theorem \ref{Theorem-DADNucDimK0}.  This family of traces can be used as a computational tool to understand the inductive limit structure of the $K_0$-group of $C^*(G^s(\Per))$.

To every $\xn \in X^u(\Per)$, we define a linear functional
\[ \tau_{\xn}: C_c(G_0(\Per)) \to \C,\qquad \tau_{\xn}(f) = \sum_{\yn \sim_0 \xn} f(\yn, \yn). \]
Observe that $\tau_{\xn}$ depends only on the equivalence class $[\xn]_0$.  Since $[\xn]_0$ is discrete, the sum defining $\tau_{\xn}(f)$ is actually a finite sum for any compactly supported $f$.  Note that the expression
\[ \tau_{\xn}(fg) = \sum_{\yn \sim_0 \xn}\sum_{\zn\sim_0\xn} f(\yn, \zn)g(\zn, \yn) \]
is symmetric in $f$ and $g$, showing that $\tau_{\xn}$ is a trace.  Additionally
\[ \tau_{\xn}(f^*f) = \sum_{\yn \sim_0 \xn}\sum_{\zn\sim_0\xn} |f(\zn, \yn)|^2 \geq 0, \]
which shows that $\tau_{\xn}$ is a positive trace and $\tau_{\xn}(f^*f) = 0$ if and only if $f$ vanishes on the equivalence class of ${\xn}$.

Let $\mathcal{F}(X^u(\Per)/{\sim_0})$ denote the vector space of (not necessarily continuous) complex-valued functions on $X^u(\Per)/{\sim_0}$.  Define
\[ \tau: C_c(G_0(\Per)) \to \mathcal{F}(X^u(\Per)/{\sim_0}),\qquad \tau(f)([\xn]_0) = \tau_{\xn}(f), \]
which is a trace that takes values in the vector space $\mathcal{F}(X^u(\Per)/{\sim_0})$.  Note that $\tau$ is positive in the sense that $\tau(f^*f)$ is a nonnegative function, and $\tau$ is faithful in the sense that $\tau(f^*f) = 0$ if and only if $f = 0$.  The function $\tau(f)$ need not be a continuous function on $X^u(\Per)/{\sim_0}$, but it retains some vestiges of continuity, as seen in Proposition \ref{Proposition-LimitOfTraces} below.

We will call a point $[\xn]_0 \in X^u(\Per)/{\sim_0}$ Hausdorff if every net in $X^u(\Per)/{\sim_0}$ that converges to $[\xn]_0$ has a unique limit.  It follows that $X^u(\Per)/{\sim_0}$ is Hausdorff if and only if every point in $X^u(\Per)/{\sim_0}$ is Hausdorff.

\begin{lemma}
	Any convergent net in $X^u(\Per)/{\sim_0}$ has at most finitely many limits.
\end{lemma}

\begin{proof}
	All limits of a convergent net $([\xn_\lambda]_0)_{\lambda \in \Lambda}$ in $X^u(\Per)/{\sim_0}$ have the same image under $r: X^u(\Per)/{\sim_0} \to Y$ by continuity and the fact that $Y$ is Hausdorff.  The result follows because $r$ is finite-to-one.
\end{proof}

Suppose $([\xn_n]_0)_{n\in\N}$ is a sequence in $X^u(\Per)/{\sim_0}$ and $([\xn_{n_k}]_0)_{k\in\N}$ is a subsequence.  Then every limit of $([\xn_n]_0)_{n\in\N}$ is also a limit of $([\xn_{n_k}]_0)_{k\in\N}$.  However, it is possible that $([\xn_{n_k}]_0)_{k\in\N}$ has limits which are not limits of the original sequence $([\xn_n]_0)_{n\in\N}$.  We shall need to consider situations that avoid this pathology.

\begin{proposition} \label{Proposition-LimitOfTraces}
	Suppose $([\xn_n]_0)_{n\in\N}$ is a convergent sequence in $X^u(\Per)/{\sim_0}$. Let $L$ denote the set of all limits of $([\xn_n]_0)_{n\in\N}$, and assume that every subsequence of $([\xn_n]_0)_{n\in\N}$ has the same set of limits.  Then for any $f \in C_c(G_0(\Per))$,
	\[ \lim_{n \to \infty}\tau_{\xn_n}(f) = \sum_{[\wn]_0 \in L}\tau_{\wn}(f). \]
\end{proposition}

\begin{proof}
	Since $r$ is continuous and $Y$ is Hausdorff, the sequence $(r([\xn_n]_0))_{n \in \N}$ converges to a unique limit $y \in Y$.  Write $r^{-1}\{y\} = \{r_1, \ldots , r_l\}$ where $r_1, \ldots , r_m \in L$ and $r_{m+1}, \ldots , r_l \notin L$.  Let $W, \{U_i\}_{i=1}^l, \{V_{\xn}\}_{\xn \in \pi_0^{-1}\{y\}}$ be as in Theorem \ref{WeakerThanCoverMapResult}.  We claim that there is a positive integer $N$ such that:
	\begin{itemize}
		\item If $1 \leq i \leq m$, then $[\xn_n]_0 \in U_i$ for all $n \geq N$.
		\item If $m+1 \leq i \leq l$, then $[\xn_n]_0 \notin U_i$ for all $n \geq N$.
	\end{itemize}
	Since each $U_i$ is a neighborhood of $r_i$, the first condition can easily be arranged by the definition of limit.  Suppose for a contradiction that the second condition is impossible to arrange.  It follows that for some $i \in \{m+1, \ldots , l\}$, the set $U_i$ contains infinitely many points in the sequence $([\xn_n]_0)_{n \in \N}$.  So we can construct a subsequence $([\xn_{n_j}]_0)_{j \in \N}$ which is contained entirely in $U_i$.  But $r$ is a local embedding, so there is an open neighborhood $U_i' \subseteq U_i$ of $r_i$ such that $r|_{U_i'}: U_i' \to r(U_i')$ is a homeomorphism.  The subsequence $(r([\xn_{n_j}]_0))_{j \in \N}$ converges to $y$ in $r(U_i')$, which means that $([\xn_{n_j}]_0)_{j \in \N}$ converges to $r_i = (r|_{U_i'})^{-1}(y)$.  This contradicts the assumptions on $([\xn_n]_0)_{n \in \N}$.  
	
	Without loss of generality, suppose $[\xn_n]_0 \in U_i$ for all $n$ if $1 \leq i \leq m$, and $[\xn_n]_0 \notin U_i$ for all $n$ if $m+1 \leq i \leq l$.  Let $\xn \in \pi_0^{-1}\{y\}$ and suppose $q(\xn) = r_i$.  If $1 \leq i \leq m$, then $(q|_{V_{\xn}}^{-1}([\xn_n]_0))_{n \in \N}$ is a lift of the sequence $([\xn_n]_0)_{n \in \N}$ to $V_{\xn}$ which converges to $\xn \in V_{\xn}$.  On the other hand, if $m+1 \leq i \leq l$, then none of the points in $V_{\xn}$ is a lift under $q$ of a point of the sequence $([\xn_n]_0)_{n \in \N}$.
	Since $[\xn_n]_0 \in U_i$ for $1 \leq i \leq m$, but not for $m+1 \leq i \leq l$, it follows that the equivalence class $[\xn_n]_0$ coincides with the set
	\[ \{(q|_{V_{\xn}})^{-1}([\xn_n]_0) \in X^u(\Per) \mid q(\xn) = r_i \text{ for some } 1 \leq i \leq m\}. \]
	
	Now let $f \in C_c(G_0(\Per))$.  It follows from the above that
	\[ \tau_{\xn_n}(f) = \sum_{\zn \sim_0 \xn_n} f(\zn, \zn) = \sum_{i=1}^m \sum_{\xn \in q^{-1}\{r_i\}} f(q|_{V_{\xn}}^{-1}([\xn_n]_0),q|_{V_{\xn}}^{-1}([\xn_n]_0)) \]
	Since $f$ has compact support, there are only finitely many $\xn \in \pi_0^{-1}\{y\}$ for which $f(\zn, \zn) \neq 0$ for some $\zn \in V_{\xn}$.  Thus, the double sum above is actually a finite sum.  Since $f$ is continuous,
	\begin{align*}
	\lim_{n \to \infty}\tau_{\xn_n}(f) &= \sum_{i=1}^m \sum_{\xn \in q^{-1}\{r_i\}} \lim_{n \to \infty}f(q|_{V_{\xn}}^{-1}([\xn_n]_0),q|_{V_{\xn}}^{-1}([\xn_n]_0))\\
	&= \sum_{i=1}^m \sum_{\xn \in q^{-1}\{r_i\}} f(\xn,\xn)\\
	&= \sum_{i=1}^m \tau_{r_i}(f)
	\end{align*}
	as desired.
\end{proof}

\begin{corollary} \label{Corollary-HausdorffTraceContinuity}
	For any $f \in C_c(G_0(\Per))$, the function $\tau(f)$ is continuous when restricted to the subset of Hausdorff points of $X^u(\Per)/{\sim_0}$.
\end{corollary}

\begin{proof}
	Immediate from the previous theorem and the fact that $X^u(\Per)/{\sim_0}$ is locally metrizable (so that sequential continuity implies continuity).
\end{proof}

Next we consider the pullbacks of these traces under the connecting homomorphism $\psi$ in our stationary inductive limit.  Recall the map $\tilde{g}: X^u(\Per)/{\sim_0} \to X^u(\Per)/{\sim_0}$ which can be thought of as a lift of the original map $g: Y \to Y$.  The pullback of a trace $\tau_{\xn}$ is given in terms of the $\tilde{g}$-preimages of $[\xn]_0$.

\begin{proposition} \label{Proposition-PullbackOfTraces}
	For any $f \in C_c(G_0(\Per))$ and $\xn \in X^u(\Per)$,
	\[ \tau_{\xn}(\psi(f)) = \sum_{[\wn]_0 \in \tilde{g}^{-1}\{[\xn]_0\}} \tau_{\wn}(f). \]
\end{proposition}

\begin{proof}
	Using the explicit description of $\psi$ given after Theorem \ref{inductiveLimit}, we have
	\begin{align*}
	\tau_{\xn}(\psi(f)) &= \sum_{\yn \sim_0 \xn}\psi(f)(\yn, \yn)\\
	&= \sum_{\yn \sim_0 \xn} f(\varphi^{-1}(\yn), \varphi^{-1}(\yn))\\
	&= \sum_{\zn \in \varphi^{-1}[\xn]_0} f(\zn, \zn)\\
	&= \sum_{[\wn]_0 \in \tilde{g}^{-1}\{[\xn]_0\}} \sum_{\vn \sim_0 \wn} f(\vn, \vn)\\
	&= \sum_{[\wn]_0 \in \tilde{g}^{-1}\{[\xn]_0\}} \tau_{\wn}(f).
	\end{align*}
\end{proof}

\section{$K$-theory computations} \label{KthComSec}

Given $[\xn]_0 \in X^u(\Per)/{\sim_0}$, let
\[ I_{\xn} = \{ f \in C_c(G_0(\Per)) \mid f(\yn, \zn) = 0 \text{ whenever } \yn \sim_0 \xn \sim_0 \zn\}, \]
which is an ideal in $C_c(G_0(\Per))$.  Let $J_{\xn}$ be the closure of $I_{\xn}$ in $C^*(G_0(\Per))$.  Observe that the quotient $C^*(G_0(\Per))/J_{\xn}$ naturally identifies with the compact operators $\K(\ell^2([\xn]_0))$, and the image of the natural homomorphism
\[ C_c(G_0(\Per))/I_{\xn} \to C^*(G_0(\Per))/J_{\xn} \cong \K(\ell^2([\xn]_0)) \]
consists of the set $\K_{\fin}(\ell^2([\xn]_0))$ of all compact operators whose matrix (with respect to the natural basis for $\ell^2([\xn]_0)$) has finitely many nonzero entries.  Further, $\tau_{\xn}$ factors through the usual trace $\tr: \K_{\fin}(\ell^2([\xn]_0)) \to \C$:
\[ \xymatrix{
	C_c(G_0(\Per)) \ar[r] \ar[rd]_-{\tau_{\xn}} & \K_{\fin}(\ell^2([\xn]_0)) \ar[d]^-{\tr}\\
	& \C
}\]
Although $\tau_{\xn}$ is not defined on all of $C^*(G_0(\Per))$, we can view it as a homomorphism defined on $K_0(C^*(G_0(\Per)))$ by Theorem \ref{Theorem-DADNucDimK0}.

\begin{proposition}
	For any $[\xn]_0 \in X^u(\Per)/{\sim_0}$, $\tau_{\xn}(K_0(C^*(G_0(\Per)))) \subseteq \Z$.
\end{proposition}

\begin{proof}
	Follows from the above discussion and the fact that $\tr: \K_{\fin}(\ell^2([\xn]_0)) \to \C$ is integer-valued on projections.
\end{proof}

If $X^u(\Per)/{\sim_0}$ is Hausdorff, then $K_0(C^*(G_0(\Per))) \cong K^0(X^u(\Per)/{\sim_0})$ by Corollary \ref{Corollary-HausdorffC*algebra}.  Here, $K_0$-classes are generated by vector bundles, and $\tau_{\xn}$ gives the dimension of the bundle at a point $[\xn]_0 \in X^u(\Per)/{\sim_0}$.  If, additionally, $X^u(\Per)/{\sim_0}$ is connected, then $[\xn]_0 \mapsto \tau_{\xn}(a)$ is continuous and integer-valued for any $a \in K_0(C^*(G_0(\Per)))$.  Hence it is constant, and by a slight abuse of notation we will denote the constant integer value by $\tau(a)$.

\begin{proposition} \label{Proposition-TraceHausdorffConnected}
	If $X^u(\Per)/{\sim_0}$ is Hausdorff and connected, then there is an integer $n \geq 2$ such that every $[\xn]_0 \in X^u(\Per)/{\sim_0}$ has exactly $n$ preimages under $\tilde{g}$, and for any $a \in K_0(C^*(G_0(\Per)))$,
	\[ \tau(\psi_*(a)) = n\cdot \tau(a). \]
\end{proposition}

\begin{proof}
	Consider the projection $1 \otimes e \in C(X^u(\Per)/{\sim_0}) \otimes \K(H) \cong C^*(G_0(\Per))$, where $e \in \K(H)$ is a rank one projection.  We know $\tau_{\xn}(1 \otimes e) = 1$ for every $\xn$.  By Proposition \ref{Proposition-PullbackOfTraces},
	\[ \tau_{\xn}(\psi(1\otimes e)) = \# \tilde{g}^{-1}\{[\xn]_0\}. \]
	But $\tau_{\xn}(\psi(1 \otimes e))$ is constant as a function of $\xn$, hence so is $\# \tilde{g}^{-1}\{[\xn]_0\} =: n$.  If $n = 1$, then $\tilde{g}$ is one-to-one.  This implies $g$ is one-to-one. Hence $g$ is a homeomorphism, which is not possible by Remark \ref{notHomeo}. Thus $n \geq 2.$
	
	For a general $a \in K_0(C^*(G_0(\Per)))$, we apply Proposition \ref{Proposition-PullbackOfTraces} again and use the fact that there are $n$ preimages to obtain
	\[ \tau(\psi_*(a)) = n\cdot \tau(a). \]
\end{proof}

\begin{theorem} \label{Theorem-TraceOnConnectedInductiveLimit} 
	Suppose $X^u(\Per)/{\sim_0}$ is Hausdorff and connected and $n$ is as in Proposition \ref{Proposition-TraceHausdorffConnected}.  Then there is an order-preserving surjective homomorphism
	\[ \tau: K_0(C^*(G^s(\Per))) \to \Z[1/n].\]
	In particular, $K_0(C^*(G^s(\Per)))$ is not finitely generated.  If $\tau: K_0(C^*(G_0(\Per))) \to \Z$ is an isomorphism, then $\tau: K_0(C^*(G^s(\Per))) \to \Z[1/n]$ is an isomorphism.
\end{theorem}

\begin{proof}
	The homomorphism $\tau: K_0(C^*(G_0(\Per))) \to \Z$ induces a map of inductive sequences:
	\[ \begin{CD}
	K_0(C^*(G_0(\Per))) @>{\psi_*}>> K_0(C^*(G_0(\Per))) @>{\psi_*}>> K_0(C^*(G_0(\Per)))@>{\psi_*}>> \ldots\\
	@V{\tau}VV @V{\tau}VV @V{\tau}VV\\
	\Z @>n>> \Z @>n>> \Z @>n>> \ldots
	\end{CD}\]
	The vertical maps are surjective and order-preserving, therefore so is the induced map on the inductive limits $\tau: K_0(C^*(G^s(\Per))) \to \Z[1/n]$.
\end{proof}

The above result can be generalized to the case where $X^u(\Per)/{\sim_0}$ is Hausdorff, but not necessarily connected.  If $a \in K_0(C^*(G_0(\Per))) \cong K^0(X^u(\Per)/{\sim_0})$, then $\tau(a)$ is a continuous integer-valued function on $X^u(\Per)/{\sim_0}$, and we have a surjection
\[ \tau: K_0(C^*(G_0(\Per))) \to C(X^u(\Per)/{\sim_0}, ~\Z). \]
Using $\tilde{g}$, define a group homomorphism
\[ \tilde{g}_*: C(X^u(\Per)/{\sim_0}, ~\Z) \to C(X^u(\Per)/{\sim_0}, ~\Z), \qquad \tilde{g}_*(h)([\xn]_0) =  \sum_{[\wn]_0 \in \tilde{g}^{-1}\{[\xn]_0\}} h([\wn]_0), \]
which is well-defined because $\tilde{g}$ is finite-to-one.  Let $\mathcal{D}(X^u(\Per)/{\sim_0}, \tilde{g})$ denote the inductive limit of the sequence of groups
\[ \begin{CD}
C(X^u(\Per)/{\sim_0}, ~\Z) @>{\tilde{g}_*}>> C(X^u(\Per)/{\sim_0}, ~\Z) @>{\tilde{g}_*}>> C(X^u(\Per)/{\sim_0}, ~\Z) @>{\tilde{g}_*}>> \ldots
\end{CD} \]

\begin{theorem} \label{Theorem-TraceOnNonconnectedInductiveLimit}
	Suppose $X^u(\Per)/{\sim_0}$ is Hausdorff.  Then there is an order-preserving surjective homomorphism
	\[ \tau: K_0(C^*(G^s(\Per))) \to \mathcal{D}(X^u(\Per)/{\sim_0}, \tilde{g}).\]
	If $\tau: K_0(C^*(G_0(\Per))) \to C(X^u(\Per)/{\sim_0}, ~\Z)$ is an isomorphism, then it follows that $\tau: K_0(C^*(G^s(\Per))) \to \mathcal{D}(X^u(\Per)/{\sim_0}, \tilde{g})$ is an isomorphism.
\end{theorem}

Using Theorem \ref{inductiveLimit}, computing the $K$-theory of $C^*(G^s(\Per))$ reduces to two steps:
\begin{enumerate}
\item Computing the $K$-theory of $C^*(G_0(\Per))$;
\item Computing the map on $K$-theory induced by $\psi$.
\end{enumerate}
Furthermore, often (when $X$ is low dimensional for example) Theorem \ref{idealStructure} can be used to complete the first of these two steps. The second step is usually more difficult.  We will illustrate how our results facilitate these computations in three examples: the $n$-solenoid, subshifts of finite type and the $aab/ab$-solenoid. The first two are very well-known and easier to handle, as $g$ is a local homeomorphism.  The latter is also well-known, see \cite{Yi} (and also \cite{Gon2, GonRamSol, ThoSol, ThoAMS}), but the computation is much more interesting.

\begin{example}[$n$-solenoid] \label{KtheoryNsole}
Consider the $n$-solenoid which arises from $g: S^1\to S^1, g(z) = z^n$.  We take $\Per = \{(1,1,1,\ldots)\}$.  Here, $X^u(\Per)$ is homeomorphic to $\R$ in such a way that $\sim_0$ is congruence mod $\Z$ and $\varphi: X^u(\Per) \to X^u(\Per)$ is a dilation by a factor of $n$.

Since $g$ is a local homeomorphism, we have $C^*(G_0(\Per)) \cong C(S^1) \otimes \K(\H)$, and consequently $K_i(C^*(G_0(\Per)) \cong K^i(S^1) \cong \Z$ for $i=0,1$.  By Proposition \ref{Proposition-TraceHausdorffConnected}, the connecting homomorphism $\psi_*: K_0(C^*(G_0(\Per))) \to K_0(C^*(G_0(\Per)))$ is simply multiplication by $n$.  Thus the map $\tau: K_0(C^*(G^s(\Per))) \to \Z[1/n]$ of Theorem \ref{Theorem-TraceOnConnectedInductiveLimit} is an isomorphism.

The elements of $K_1(C^*(G_0(\Per)))$ can all be represented by $S^1$-valued functions on the unit space $X^u(\Per)$.  The isomorphism $K_1(C^*(G_0(\Per))) \cong \Z$ is given by the winding number.  Since $\varphi$ is an orientation-preserving dilation, we see that $\psi$ preserves winding numbers.  So $\psi_* = \id$ on $K_1(C^*(G_0(\Per))$, and consequently $K_1(C^*(G^s(\Per))) \cong \Z$.
\end{example}

\begin{example}[Subshifts of finite type] \label{KtheorySFT}
	Here $Y = \Xi$ is a one-sided shift space (a Cantor set), and $g: \Xi \to \Xi$ is the shift map, which is a local homeomorphism (see \cite{Wie}).  So $K_0(C^*(G_0(\Per))) \cong K^0(\Xi)$ and the trace $\tau: K_0(C^*(G_0(\Per))) \to C(\Xi, \Z)$ is an isomorphism.  It follows from Theorem \ref{Theorem-TraceOnNonconnectedInductiveLimit} that $K_0(C^*(G^s(\Per))) \cong \mathcal{D}(\Xi, g)$.  We leave it to the interested reader to compute $\mathcal{D}(\Xi, g)$ and reconcile it with the well-known computation of the dimension group associated to a shift of finite type, see for example \cite[Chapter 7]{LM} or \cite[Chapter 3]{PutHom}.  We clearly have $K_1(C^*(G^s(\Per))) = 0$ because $K^1(\Xi) = 0$.
\end{example}

Next we will compute the $K$-theory groups for the stable $C^*$-algebra associated to the $aab/ab$-solenoid.  The techniques used in this example can be generalized to any one-dimensional Williams solenoid. We note that the $K$-theory of such solenoids has also been computed in \cite{ThoSol, WilPhD, Yi}.

\begin{example}[$aab/ab$-solenoid] Consider the ideal $J$ obtained from the open set $U_a \cup U_b$ via Theorem \ref{idealStructure} where $U_a$ and $U_b$ are the two open sets in Example \ref{coverAABABcase}; each is obtained by taking a circle in $Y$ and removing $p$. This ideal is the completion of the space of functions that vanish on the equivalence classes of the three non-Hausdorff points in $X^u(\Per)/{\sim_0}$. Note that $J$ is the $C^*$-algebra of the groupoid $G_0(\Per)|_{U_a \cup U_b}$. We have a short exact sequence:
\[
0 \rightarrow J \rightarrow C^*(G_0(\Per)) \rightarrow C^*(G_0(\Per)) / J \rightarrow 0
\]
By Theorem \ref{WeakerThanCoverMapResult}, $G_0(\Per)|_{U_a \cup U_b}$ is the groupoid of an equivalence relation induced by a covering map with Hausdorff quotient $U_a \sqcup U_b$.  It follows that $J = C^*(G_0(\Per)|_{U_a \cup U_b}) \cong C_0(U_a \sqcup U_b) \otimes \K$. Each of $U_a, U_b$ is an open interval, so our short exact sequence becomes
\[
0 \rightarrow (C_0(\R)\otimes \mathcal{K}) \oplus (C_0(\R)\otimes \mathcal{K}) \rightarrow C^*(G_0(\Per)) \rightarrow \mathcal{K} \oplus \mathcal{K} \oplus \mathcal{K} \rightarrow 0 
\] 
Each $\K$ in the quotient corresponds to one of the non-Hausdorff points $ba, ab, aa$.  Applying $K$-theory and using
\[ K_0(C_0(\R)) \cong \{0 \}, \ K_1(C_0(\R)) \cong \Z, \ K_0(\mathcal{K}) \cong \Z, \ K_1(\mathcal{K})\cong \{0 \}
\]
we obtain the following six term exact sequence
\begin{center}
$\begin{CD}
0 @>>> K_0(C^*(G_0(\Per))) @>>> \Z \oplus \Z \oplus \Z \\
@AAA @. @V{\delta_0} VV \\
0 @<<<  K_1(C^*(G_0(\Per))) @<<< \Z \oplus \Z 
\end{CD}$
\end{center}
Hence
\[ 
K_0(C^*(G_0(\Per))) \cong {\rm ker}(\delta_0) \hbox{ and }K_1(C^*(G_0(\Per))) \cong {\rm coker}(\delta_0)
\]
and thus we need only to compute the boundary map $\delta_0$ (the exponential map). 

To do so, it is useful to label the copies of $\Z$ with the relevant generators. We will write $K_0$ of the quotient algebra as $\Z_{ba}\oplus \Z_{ab} \oplus \Z_{aa}$.  Each generator is a rank $1$ projection, and the injection $K_0(C^*(G_0|\Per))) \to \Z_{ba}\oplus \Z_{ab} \oplus \Z_{aa}$ from the diagram is given in terms of the traces of Section \ref{Section-Traces} by $\tau_{ba} \oplus \tau_{ab} \oplus \tau_{aa}$.  In particular, these three traces completely detect the $K_0$-group of $C^*(G_0(\Per))$.

We write $K_1(J)$ as $\Z_{U_a} \oplus \Z_{U_b}$.  All $K_1$ elements can be represented by $S^1$-valued functions on the unit space, and the two integers correspond to winding numbers associated to the two different intervals $a$ and $b$.

Now we compute $\delta_0$.  Begin with the generator of $\Z_{ba}$.  Let $\xn \in X^u(\Per)$ be an element with $[\xn]_0 = ba$.  We can take a positive continuous bump function $f$ on the unit space $X^u(\Per)$, supported over a very small interval, which is $1$ at $\xn$ to be a self-adjoint lift of the generator of $\Z_{ba}$.  Then $\exp(2\pi if)$ has winding number $1$ around $b$ and $-1$ around $a$.  That is, $\delta_0(1,0,0) = (-1,1)$.  Similarly, one shows $\delta_0(0,1,0) = (1,-1)$ and $\delta_0(0,0,1) = (0,0)$.  Hence
\[
\delta_0 = \begin{bmatrix} -1 & 1 & 0 \\ 1 & -1 & 0 \end{bmatrix}: \Z_{ba} \oplus \Z_{ab} \oplus \Z_{aa} \rightarrow \Z_{U_a} \oplus \Z_{U_b}.
\]
It follows that
\[
K_0(C^*(G_0(\Per))) \cong {\rm ker}(\Phi) \cong \Z \oplus \Z \hbox{ and }K_1(C^*(G_0(\Per))))\cong {\rm coker}(\Phi)\cong \Z.
\]
As generators of $K_0(C^*(G_0(\Per)))$, we take $\alpha = (1,1,0)$ and $\beta = (0,0,1)$.  Observe that $\tau_{ba} = \tau_{ab}$ on $K_0$ because they are equal on $\alpha$ and $\beta$.  Hence, $K_0$ is completely detected by just $\tau_{ba}$ and $\tau_{aa}$.  

Next, we must compute the map on $K$-theory obtained from the map $\psi$ in the inductive limit.  For $K_0$, we use the results of Section \ref{Section-Traces}.  Using Corollary \ref{Corollary-HausdorffTraceContinuity}, we see that the homomorphism $\tau_{\xn}: K_0(C^*(G_0(\Per))) \to \Z$ does not depend on the choice of $[\xn]_0 \in U_a$ by continuity and the fact that it is integer-valued.  We will refer to this homomorphism as simply $\tau_a$.  Similarly, we have $\tau_b$ for any point on $U_b$.

Next, consider a sequence of points $([\xn_n]_0)_{n \in \N} \in U_a$ on the left half of $U_a$ which converge towards the place where the three non-Hausdorff points are.  This sequence has two limits, namely $ba$ and $aa$.  It follows from Proposition \ref{Proposition-LimitOfTraces} that (on $K_0$)
\[ \tau_a = \lim_{n \to \infty} \tau_a = \lim_{n \to \infty} \tau_{\xn_n} = \tau_{ba} + \tau_{aa}. \]
Similarly, by considering a sequence on the right half of $U_a$ that comes in from the other side, we get
\[ \tau_a = \tau_{ab} + \tau_{aa}. \]
This tells us how to evaluate $\tau_a$, but it also gives a second proof that $\tau_{ba} = \tau_{ab}$ on $K_0$.  By similarly considering sequences on $U_b$, we obtain
\[ \tau_b = \tau_{ab} = \tau_{ba}. \]

Now we can compute $\psi_*$ on $K_0$ using Proposition \ref{Proposition-PullbackOfTraces} and the results of Example \ref{Example-aab/ab-InducedMapOnQuotient}.  The point $ba$ has three pre-images under $\tilde{g}$, so
\[ \tau_{ba} \circ \psi_* = \tau_{ba} + \tau_{ab} + \tau_{aa} = 2\tau_{ba} + \tau_{aa}. \]  The point $aa$ has one pre-image, which is on $U_a$, so
\[ \tau_{aa} \circ \psi_* = \tau_a = \tau_{ba} + \tau_{aa}. \]
(Although it is redundant, we can also compute $\tau_{ab} \circ \psi_* = \tau_a + \tau_b = 2\tau_{ba} + \tau_{aa}$.)  The traces $\tau_{ba}$ and $\tau_{aa}$ determine the entire $K_0$-group, so we can use this to determine
\[ \psi_*(\alpha) = 2\alpha + \beta,\qquad \psi_*(\beta) = \alpha + \beta. \]
So on $K_0(C^*(G_0(\Per))) \cong \Z \oplus \Z$, $\psi_*$ is given by $\begin{bmatrix}
2 & 1\\
1 & 1
\end{bmatrix}$.  Since $\psi_*$ is an automorphism, we conclude that $K_0(C^*(G^s(\Per))) \cong \Z \oplus \Z$.

The image of either generator of $\Z_{U_a} \oplus \Z_{U_b}$ generates $K_1(C^*(G_0(\Per))) \cong \Z$.  The integer value is again the winding number.  An argument similar to that of the $n$-solenoid shows that $\psi_*$ is the identity on $K_1(C^*(G_0(\Per)))$, so $K_1(C^*(G^s(\Per)))\cong \Z$.
\end{example}

Since the $K_0$-group for the $aab/ab$-solenoid is finitely generated, Theorem \ref{Theorem-TraceOnConnectedInductiveLimit} explains why we necessarily had a non-Hausdorff quotient $X^u(\Per)/{\sim_0}$.  Consequently, we see that this Smale space $X$ is not conjugate to any Wieler solenoid of the form $\varprojlim (Y, g)$ where $g: Y \to Y$ is a local homeomorphism and $Y$ is connected. Although it is likely that this result is known to experts, we feel that our method (which is $K$-theoretic) is a particularly nice way of showing that certain Smale spaces cannot be written in the form $\varprojlim (Y, g)$ where $g: Y \to Y$ is a local homeomorphism and $Y$ is connected.

\appendix

\section{$C^*$-stability results} \label{Section-C*-Stability}
The setup of this appendix is quite different from the rest of the paper. Here $(X, \varphi)$ is a mixing Smale space without any assumptions on the stable sets. We recall that a $C^*$-algebra $A$ is called $C^*$-stable if $A \cong A \otimes \mathcal{K}$, where $\mathcal{K}$ is the $C^*$-algebra of compact operators on a separable infinite-dimensional Hilbert space. This notion is usually referred to as ``stable", but we prefer ``$C^*$-stable" to avoid confusion with the dynamical term. 

The goal of this appendix is to prove that $C^*(G^s(\Per))$ and $C^*(G_0(\Per))$ are $C^*$-stable. Note that the latter algebra only exists in the special case when the stable sets are totally disconnected. The general plan is to use Theorem 2.1 and Proposition 2.2 of \cite{HjeRor}. We learned of this type of argument from the proof of \cite[Lemma 4.15]{ThoAMS}. 

\begin{lemma} \label{etaleGrpLemma}
	Suppose that $\mathcal{G}$ is an \'etale groupoid and for each $f\in C_c(\mathcal{G})$ there exists $v\in C_c(\mathcal{G})$ such that
	\[
	v^*v f= f \hbox{ and }fv=0.
	\]
	Then any $C^*$-completion of $C_c(\mathcal{G})$ is $C^*$-stable.
\end{lemma}
\begin{proof}
	Let $a$ be any positive element in $C^*(\mathcal{G})$ and $\epsilon>0$. There exists $f \in C_c(\mathcal{G})$ such that 
	\[
	\|a - f^*f \| < \epsilon
	\]
	and take $v \in C_c(\mathcal{G})$ as in the statement of the lemma. 
	
	In the context of condition (b) in Proposition 2.2 of \cite{HjeRor}, consider 
	\[
	b= f^*v^*v f = f^* f \hbox{ and }c= v f f^* v.
	\] 
	Hence, we have that $b \sim c$ (in the notation of \cite{HjeRor}) and 
	\[ 
	bc = f^* f v f f^* v = 0.
	\]
	Proposition 2.2 and Theorem 2.1 of \cite{HjeRor} can now be applied and give the result.
\end{proof}

\begin{theorem} \label{SandUstable}
	Suppose $(X, \varphi)$ is a mixing Smale space and $\Per$ is a finite $\varphi$-invariant set. Then $C^*(G^s(\Per))$ and $C^*(G_u(\Per))$ are $C^*$-stable.
\end{theorem}
\begin{proof}

	Since the stable Smale space algebra of $(X, \varphi^{-1})$ is the unstable algebra of $(X, \varphi)$, the proof will be complete upon showing $C^*(G^s(\Per))$ is stable. Furthermore, we can assume $\Per$ satisfies $X^u(p) \cap X^u(p') = \emptyset$ for $p \neq p' \in \Per$. Otherwise, we could replace $\Per$ with $\Per - \{ p' \}$ without changing the groupoid. (Note that if $X^u(p) \cap X^u(p') \neq \emptyset$, then $X^u(p)=X^u(p')$.)
	
	Our goal is to apply the previous lemma. Let $f \in C_c(G^s(\Per))$, $K_r := r( {\rm supp}(f))$ and $K_s:=s({\rm supp}(f))$. Since these are both compact subsets of $X^u(\Per)$ there exists $N\in \N$ such that $\varphi^{-N}(K_r)$ and $\varphi^{-N}(K_s)$ are contained in $X^u\left(\Per, \frac{\epsilon_X}{2}\right)$. By proceeding inductively, one can obtain $\{x_p\}_{p\in \Per}$ such that
	\begin{enumerate}
		\item $x_p \in X^u(p) \cap X^s(p, \frac{\epsilon_X}{2})$;
		\item $X^u(x_p, \epsilon_X) \cap X^u(\Per, \epsilon_X)= \emptyset$;
		\item $X^u(x_p, \epsilon_X) \cap X^u(x_{p'}, \epsilon_X) = \emptyset$ for each $p \neq p' \in \Per$.
	\end{enumerate} 
	
	Define $V_{p, x_p} := \{ ([ x , x_p ], x) \mid x \in X^u(p, \epsilon_X) \}$. The sets $V_{p, x_p}$ are open (see Example \ref{exOpenSet}) and disjoint. Let
	\[ 
	V:=\bigcup_{p\in P} V_{p, x_p}.
	\]
	Then, $\varphi^{-N}(K_r) \subseteq s(V)$ and by construction (i.e., item (2), the fact that $\varphi^{-N}(K_s) \subseteq X^u\left(P, \frac{\epsilon_X}{2}\right)$, etc), $\varphi^{-N}(K_s)\cap r(V) = \emptyset$.
	
	Let $w \in C_c(G^s(P))$ with support in $V$ and $v(y, x)=1$ for $x \in \varphi^{-N}(K_r)$. Letting $\alpha$ denote the action on $C^*(G^s(\Per))$ induced from $\varphi$,  it follows that $w^*w \alpha^N(f) = \alpha^N(f)$.
	
	Let $v= \alpha^{-N}(w)$. Then $f^*f = f^* v^* v f$. Moreover, the condition $\varphi^{-N}(K_s)\cap r(V) = \emptyset$ implies that $\alpha^N(f) w=0$ and hence that $f \alpha^{-N}(w)=f v=0$. Thus $v$ satisfies the conditions of the previous lemma; it implies the required result.
\end{proof}

The next result follows from the previous theorem, \cite[Theorems 3.7 and 4.2]{PutSpi} and Brown's theorem \cite{Bro}.
\begin{corollary} 
	Suppose $(X, \varphi)$ is a mixing Smale space and $\Per$ and $\Per'$ are finite $\varphi$-invariant subset of $X$. Then $C^*(G^s(\Per)) \cong C^*(G^s(\Per'))$ and $C^*(G_u(\Per)) \cong C^*(G_u(\Per'))$.
\end{corollary}

\begin{theorem}
	Suppose $(Y, d_Y, g)$, $K$, and $\beta$ are as in Definition \ref{WielerAxioms}, $(X, \varphi)$ denotes the associated Wieler solenoid and $\Per$ is a finite $\varphi$-invariant set. Then $C^*(G_0(\Per))$ is $C^*$-stable.
\end{theorem}
\begin{proof}
	Let $f \in C_c(G_0(\Per))$, $K_r := r( {\rm supp}(f))$ and $K_s:=s({\rm supp}(f))$. Since $K_r$ is compact, there exists a finite subcover $\{ X^u(x_i, \delta_{x_i}) \}_{i=1}^n$. 
	
	Using induction and Proposition \ref{infiniteDiscreteSet}, we can find $(x^{\prime}_i)_{i=1}^n$ with the following properties:
	
	\begin{enumerate}
		\item $x_i \sim_0 x^{\prime}_i$;
		\item The sets $V_{x_i, x^{\prime}_i, h_i} := \{ (h_i(z), z) \mid z \in X^u(x_i, \delta_{x_i}) \}$ ($i=1, \ldots n$) satisfy
		\begin{enumerate}
			\item $r(V_{x_i, x^{\prime}_i, h_i} ) \cap K_s = \emptyset$ for each $i=1, \ldots, n$;
			\item $r(V_{x_i, x^{\prime}_i, h_i} ) \cap r(V_{x_j, x^{\prime}_j, h_j} )= \emptyset$ for $i\neq j$.
		\end{enumerate}
	\end{enumerate}
	
	For each $i=1, \ldots, n$ define $v_i \in C_c(G_0(\Per))$ such that 
	\begin{enumerate}
		\item ${\rm supp}(v_i) \subset V_{x_i, x^{\prime}_i, h_i}$;
		\item $(\sum_{i=1}^n v_i^*v_i )|_{K_r} =1$
	\end{enumerate}
	Since $r(V_{x_i, x^{\prime}_i, h_i} ) \cap r(V_{x_j, x^{\prime}_j, h_j} )= \emptyset$ for $i\neq j$, we have that $v_j^*v_i=0$ for $i\neq j$. Let $v= \sum_{i=1}^n$. Then, 
	\[
	v^*v = \left( \sum_{j=1}^n v_j^* \right) \left(\sum_{i=1}^n v \right) = \sum_{i=1}^n v_i^*v_i.
	\]
	Using the fact that $(\sum_{i=1}^n v_i^*v_i )|_{K_r} =1$, we have that $v^*v f = (\sum_{i=1}^n v_i^*v_i)f = f$. 
	
	Furthermore, since $r(V_{x_i, x^{\prime}_i, h_i} ) \cap K_s = \emptyset$ for each $i=1, \ldots n$, $f v_i =0$ for each $i=1, \ldots n$. Thus $fv=0$ and Lemma \ref{etaleGrpLemma} can be applied to obtain the result.
	
\end{proof}


\end{document}